\documentclass[12pt,reqno,a4paper]{amsart}
\usepackage{amsmath,amssymb,amsfonts,amscd}
\usepackage[mathscr]{eucal}
\usepackage{comment}
\usepackage{color}
\usepackage{hyperref}

\date{16 (29) August 2018}  

\author{Hovhannes~Khudaverdian}
\author{Theodore~Voronov}
\address{School of Mathematics, University of Manchester, Manchester, M13 9PL, UK}
\email{khudian@manchester.ac.uk}
\address{{School of Mathematics,  University of Manchester,    Manchester,   M13 9PL,  UK}}
\email{theodore.voronov@manchester.ac.uk}
\address{%
{Faculty of Physics, Tomsk State University, Tomsk, 634050, Russia}}

\title[Thick morphisms,  higher Koszul brackets, and $\Linf$-algebroids]{Thick morphisms,  higher Koszul brackets, and $\Linf$-algebroids}

\newcommand{\ted}[1]{\textcolor{blue}{#1}}  

\newtheorem{theorem}{Theorem}
\newtheorem{proposition}{Proposition}

\newtheorem{corollary}{Corollary}
\theoremstyle{definition}
\newtheorem{definition}{Definition}
\newtheorem{example}{Example}
\newtheorem{remark}{Remark}

\def\co{\colon\thinspace}

\renewcommand{\geq}{\geqslant}

\newcommand{\infto}{\rightsquigarrow}

\DeclareMathOperator{\fun}{\mathit{C^{\infty}}}
\DeclareMathOperator{\funn}{\mathbf{C^{\infty}}}
\DeclareMathOperator{\pfunn}{\mathbf{\Pi\!C^{\infty}}}

 \DeclareMathOperator{\sign}{sgn}

\newcommand{\der}[2]{{\frac{\partial {#1}}{\partial {#2}}}}

\newcommand{\lder}[2]{{\partial {#1}/\partial {#2}}}
\newcommand{\dder}[3]{{\frac{\partial^2 {#1}}{\partial {#2}\partial {#3}}}}

\newcommand{\Z}{{\mathbb Z_{2}}}
\newcommand{\ZZ}{{\mathbb Z}}

\renewcommand{\a}{\alpha}
\renewcommand{\b}{\beta}
\newcommand{\e}{\varepsilon}
\newcommand{\s}{\sigma}
\newcommand{\f}{{\varphi}}
\renewcommand{\O}{\Omega}

\renewcommand{\o}{\omega}
\newcommand{\g}{{\gamma}}

\newcommand{\h}{\eta}
\newcommand{\io}{\iota}

\newcommand{\F}{{\Phi}}

\newcommand{\la}{{\lambda}}
\newcommand{\x}{{\xi}}
\renewcommand{\d}{\delta}

\renewcommand{\L}{{\Lambda}}

\newcommand{\ft}{{\tilde f}}
\newcommand{\jt}{{\tilde j}}

\newcommand{\itt}{{\tilde \imath}}

\newcommand{\at}{{\tilde a}}
\newcommand{\bt}{{\tilde b}}
\newcommand{\ct}{{\tilde c}}
\newcommand{\ut}{{\tilde u}}
\newcommand{\vt}{{\tilde v}}

\newcommand{\wt}{{\tilde w}}

\newcommand{\Xt}{{\tilde X}}
\newcommand{\At}{{\tilde A}}
\newcommand{\Bt}{{\tilde B}}

\newcommand{\Ht}{{\tilde H}}

\newcommand{\Pt}{{\tilde P}}
\newcommand{\Rt}{{\tilde R}}

\newcommand{\lat}{{\tilde \lambda}}

\newcommand{\La}{\Lambda}

\DeclareMathOperator{\Vect}{\mathrm{Vect}}

\DeclareMathOperator{\Mult}{\mathfrak{A}}

\newcommand{\lsch}{{[\![}}
\newcommand{\rsch}{{]\!]}}

\newcommand{\Li}{L_{\infty}}
\newcommand{\Pin}{P_{\infty}}
\newcommand{\Si}{S_{\infty}}

\newcommand{\Sinf}{S_{\infty}}
\newcommand{\Pinf}{P_{\infty}}
\newcommand{\Linf}{L_{\infty}}

\newcommand{\att}{{\tilde{\a}}}
\newcommand{\btt}{{\tilde{\b}}}
\newcommand{\gtt}{{\tilde{\g}}}

\newcommand{\tto}{{\linethickness{2pt}
		  \,\begin{picture}(1,0)
                   \put(0,0.26){\line(1,0){0.95}}
                   \put(0,0){$\boldsymbol{\rightarrow}$}
                  \end{picture}
                  }\,
}

\newcommand{\oto}{{\linethickness{0.5pt}
		  \,\begin{picture}(1,0)
		  \put(0.07,0.175){\line(0,1){0.2}}
                   \put(-0.01,0){$\boldsymbol{\Rightarrow}$}
                  \end{picture}
                  }\,
}

\newcommand{\specto}{\rightsquigarrow}

\newcommand{\mul}{P}

\newcommand{\Sch}{\mathfrak{S}}
\newcommand{\Poi}{\mathfrak{P}}
\newcommand{\void}{\varnothing}

\newcommand{\kir}{\boldsymbol{\kappa}}

\begin{document}

\begin{abstract}
It is a classical fact in Poisson geometry that the cotangent bundle of a
Poisson manifold has the structure of a Lie algebroid. Manifestations
of this structure are the Lichnerowicz differential on multivector fields (calculating Poisson cohomology) and  the  Koszul bracket  of differential forms. ``Raising indices'' by the Poisson tensor maps the de Rham differential to the  Lichnerowicz differential and the Koszul bracket to the Schouten bracket. In this paper, we present a homotopy analog of the above results. When an ordinary Poisson structure is replaced by a homotopy one, instead of a single Koszul bracket there arises an infinite sequence of ``higher Koszul brackets'' 
defining an $L_{\infty}$-algebra structure on forms (Khudaverdian--Voronov  \href{https://arxiv.org/abs/0808.3406}{\texttt{arXiv:0808.3406}}). We show how to construct a non-linear transformation, which is an $L_{\infty}$-morphism, from this $L_{\infty}$-algebra 
to the Lie superalgebra of multivector fields with the canonical Schouten bracket. This is done by using the new notion of ``thick morphisms'' of supermanifolds recently introduced (see \href{http://arxiv.org/abs/1409.6475}{\texttt{arXiv:1409.6475}} and \href{https://arxiv.org/abs/1411.6720}{\texttt{arXiv:1411.6720}}).  
\end{abstract}

\maketitle

\tableofcontents

\section{Introduction}
It is a classical fact of Poisson geometry that a Poisson structure on a manifold $M$ endows the algebra of differential forms on $M$ with a bracket operation whose properties are similar to those of the canonical Schouten   bracket on multivector fields. Under this operation and the usual multiplication, the space of differential forms becomes an odd Poisson algebra. It was first introduced by Koszul~\cite{koszul:crochet85} and is known as the \emph{Koszul bracket}. (Kosmann-Schwarzbach and Magri~\cite{yvette:magri-pn1990}, who  gave for for it an alternative construction refer  to it as the `Koszul--Schouten bracket'.)  The Koszul bracket is odd in the sense of parity and also has degree $-1$ with respect to the usual grading of forms. In particular, the space of $1$-forms is closed under this bracket and thus forms a Lie algebra. As pointed out in~\cite[footnote 5 on p. 6]{yvette:modular2008},  such a bracket of $1$-forms was defined (in the symplectic case) already in Abraham--Marsden~\cite{abraham:marsden1967}; it was re-discovered in the context of soliton theory by~Fuchssteiner~\cite{fuchssteiner82}, Dorfman~\cite{dorfman:1984},  Daletsky~\cite{daletsky:1984},  and Magri--Morosi~\cite{magri:morosi1984} (see  also  Dorfman~\cite[\S 2.7]{dorfman:book}).

Besides the Koszul bracket, a Poisson structure $P$ on a manifold $M$ defines another geometric object, the \emph{Lichnerowicz differential} $d_P$ discovered in~\cite{lichner:poisson1977}. It acts on multivector fields by the formula $d_P :=\lsch P,-\rsch$, where $P$ is the Poisson tensor and $\lsch -,-\rsch$ denotes the canonical Schouten bracket.

The crucial fact concerning   the Koszul bracket and Lichnerowicz differential is as follows (see e.g. Karas\"{e}v--Maslov~\cite[Ch.~I,Thm.~2.6]{karasev:maslov1991}). Consider the   diagram
\begin{equation}\label{eq.diagram}
    \begin{CD} \Mult^k (M)@>{d_P}>> \Mult^{k+1}(M)\\
                @A{P^{\#}}AA         @AA{P^{\#}}A\\
                \O^k(M)@>{d}>> \O^{k+1}(M) \,.
    \end{CD}
\end{equation}
where the vertical arrows ${P^{\#}}$ are   linear maps given by `raising indices' with the help of  the Poisson tensor $P$. We denote by    $\Mult^k(M)$   the space  of multivector fields on
$M$ of degree $k$.
Then: (a) the diagram is commutative, so that the de Rham differential on forms is transformed to the Lichnerowicz differential on multivector fields; and, (b) the Koszul bracket on forms is mapped to the   Schouten bracket on multivector fields.
(Note that here we have   two canonical structures on $M$, the de Rham differential and the Schouten bracket, and  two structures defined by a Poisson tensor, the Lichnerowicz differential and the Koszul bracket. The maps ${P^{\#}}$ exchange them crosswise.) In fact, the equality
\begin{equation}
    P^{\#}[\o,\s]_P = \lsch P^{\#}(\o),P^{\#}(\s)\rsch\,,
\end{equation}
where $[-,-]_P$ denotes the Koszul bracket of forms, can be used as its definition: if one assumes the invertibility of  $P^{\#}$, i.e. that the structure is symplectic, then one may set
\begin{equation*}
    [\o,\s]_P := (P^{\#})^{-1}\lsch P^{\#}(\o),P^{\#}(\s)\rsch\,,
\end{equation*}
and then the calculation shows that the inverse matrix $P_{ab}$ to the Poisson tensor $P^{ab}$ disappears from the explicit formula~\cite{tv:higherpoisson}.

In this paper, we are concerned with the \textbf{homotopy analog} of the above. One obtains a `homotopy relaxing' of a Poisson algebra structure if the usual Jacobi identity is satisfied only up to an exact term or more precisely, if its left-hand side is only homotopic to zero (in the algebraic sense). Such a situation may occur   when a Poisson bracket is lifted from homology classes to the chain level. We give precise definitions in the next section, but here it is sufficient to say that a \emph{homotopy Poisson structure},
instead of a single Poisson bracket with two arguments as in the usual case, consists of a possibly infinite sequence of (multi)brackets $\{-,\ldots,-\}$ with $n$ arguments, $n=0,1,2,3,\ldots$, which are multilinear, antisymmetric and satisfy an infinite sequence of  `generalized Jacobi identities'. In particular, the unary bracket $\{-\}$ is a differential in the sense of homological algebra (more precisely, it is so if the $0$-bracket vanishes), and in the generalized Jacobi identity with three arguments, the ternary bracket plays the role of an algebraic homotopy. Other higher brackets are interpreted as `higher homotopies'. For this to make proper sense, one should assume that $M$ is a supermanifold  instead of an ordinary manifold. Antisymmetry in particular is understood in the $\Z$-graded sense. For a homotopy Poisson structure, the place of a Poisson tensor is taken by an even function $P$ on the odd cotangent bundle $\Pi T^*M$ satisfying $\lsch P,P\rsch=0$ (which with abuse of language we shall still refer to as a `Poisson tensor'), so that `higher Poisson brackets' are defined by
\begin{equation*}
    \{f_1,\ldots,f_n\}_P:= \pm \lsch \ldots \lsch P,f_1\rsch,\ldots,\rsch |_M\,.
\end{equation*}
In fact, only the power expansion of $P$ around the zero section of $\Pi T^*M$ matters, so one can see $P$ as a formal sum of homogeneous terms (which can be identified with multivectors of a particular degree):
\begin{equation*}
    P=P_0+P_1+P_2+P_3+\ldots\,,
\end{equation*}
where the classical case corresponds to $P=P_2$. 

What becomes of the above picture in the homotopy situation?

The analog of the Lichnerowicz differential $d_P$ makes sense and is defined by the same formula: $d_P=\lsch P, -\rsch$. The only difference is that the  operator $d_P$, which is odd in the sense of parity, is no longer homogeneous in the usual grading of multivectors.  We  regard it as an operator $\Mult(M)\to \Mult(M)$, where now $\Mult(M):=\fun(\Pi T^*M)$. The identity $d_P^2=0$ holds by the Jacobi identity for the Schouten bracket and the condition $\lsch P,P\rsch=0$.

In~\cite{tv:higherpoisson}, we showed that a homotopy Poisson structure on a supermanifold $M$ induces an infinite  sequence of odd brackets $[-,\ldots,-]_P$ on the space of forms on $M$, which we called \emph{higher Koszul brackets}. Namely, for all $n=0,1,2,\ldots$\,, we set (up to signs)
\begin{align*}
    [df_1,\ldots,df_n]_P := \pm d\{f_1,\ldots,f_n\}_P\,, \\
    [df_1,\ldots,df_n,g]_P :=\pm \{f_1,\ldots,f_n,g\}_P\,,
\end{align*}
and all other higher Koszul  brackets are essentially defined from these  ones by the Leibniz rule. (More details in the main text.) 

As shown in~\cite{tv:higherpoisson}, it is also possible to define  an analog of the raising indices map $P^{\#}$, which we shall denote here $a_P^*$,   and now there is a commutative diagram
\begin{equation}\label{eq.diagram2}
    \begin{CD} \Mult  (M)@>{d_P}>> \Mult(M)\\
                @A{a_P^*}AA         @AA{a_P^*}A\\
                \O(M)@>{d}>> \O(M) \,,
    \end{CD}
\end{equation}
analogous to~\eqref{eq.diagram}. So the analog of (a) above holds. Here   $a_P^*\co \O(M)\to \Mult(M)$ is an algebra homomorphism induced by some vector bundle morphism $a_P\co \Pi T^*M\to \Pi TM$. In particular, $a_P^*$ is linear. But here comes the crucial difference with the classical situation. While there is   a single canonical Schouten bracket on $\Mult(M)$, now, unlike   the classical case,  there is a whole   infinite sequence of higher Koszul brackets on $\O(M)$. Hence it is impossible for a linear transformation like $a_P^*$ to take (many) Koszul brackets  to (one) Schouten bracket, so the analog of (b) fails. Therefore it was conjectured in~\cite{tv:higherpoisson} that \emph{there must exist an $\Li$-morphism transforming higher Koszul brackets to the Schouten bracket}. This should be an essentially \textbf{non-linear} transformation $\O(M)$ to $\Mult(M)$! (Recall that an $\Li$-morphism of $\Li$-algebras can be described either as as sequence of multilinear maps or as a nonlinear mapping. More details will be provided in the main text.) So the question was how to find such a non-linear map of forms to multivectors.

In this paper we solve this problem. The solution uses a new technique of \emph{thick} or \emph{microformal} morphisms of supermanifolds introduced in~\cite{tv:nonlinearpullback,tv:microformal}. (An indication of the solution was already given there. Here we provide full details.) One of the purposes of this paper is to demonstrate the power of the new technique and to show that it is indispensable for homotopy Poisson structures, $\Li$-algebras and $\Li$-algebroids.

Let us elaborate on the latter point. It is known that a `correct' geometric framework for   the (classical) Koszul bracket and   Lichnerowicz differential  is that of Lie algebroids (see~\cite[Ch.~10]{mackenzie:book2005}. A Poisson structure (ordinary) on a manifold $M$ makes the cotangent bundle $T^*M$ a Lie algebroid. Raising indices   by a Poisson tensor $P^{ab}$ is exactly the anchor map $T^*M\to TM$ and the bracket of sections of $T^*M$ is the Koszul bracket of $1$-forms. The Koszul bracket on the whole algebra of forms is a particular case of a Lie--Schouten bracket defined for an arbitrary Lie algebroid $A\to M$, while the Lichnerowicz differential $d_P\co \Mult^{k}(M)\to \Mult^{k+1}(M)$ is a particular case of a differential calculating the cohomology of a Lie algebroid. See~Mackenzie~\cite[Ch.7]{mackenzie:book2005} and references therein;  particularly, Kosmann-Schwarzbach and Magri~\cite{yvette:magri-pn1990}. One can show, conversely, that a Poisson structure on $M$ is equivalent to equipping the cotangent bundle $T^*M$ with a Lie algebroid structure `linear' in a particular sense. The commutativity of the diagram~\eqref{eq.diagram} and the relation between the Koszul bracket of forms induced by a Poisson structure and the Schouten bracket of multivector fields can be seen as special cases of   properties of arbitrary Lie algebroids. 

In the homotopy case, we showed in~\cite{tv:higherpoisson} that the higher Koszul brackets are a manifestation of an $\Li$-algebroid structure 
in $T^*M$ induced by a Poisson structure on $M$. Respectively, the commutative diagram~\eqref{eq.diagram2} fit into this framework. The same is true for the  problem of linking higher Koszul brackets with the Schouten brackets. Both the problem and the solution presented here by the method of thick morphisms   naturally belong to the general theory of $\Li$-algebroids (as we shall explain).  

Furthermore, for the classical case (i.e. for an ordinary Poisson structure), the pair of dual bundles  $(TM, T^*M)$ is in fact a Lie bialgebroid. Mackenzie and Xu~\cite{mackenzie:bialg}, see also~\cite[Ch.\,12]{mackenzie:book2005}, introduced the notion of a triangular Lie bialgebroid. It contains the example of $(TM, T^*M)$ for a Poisson manifold $M$ as well as Drinfeld's triangular Lie bialgebras. We can show that for a homotopy Poisson structure on a supermanifold $M$, higher Koszul brackets naturally fit into   an abstract framework of \emph{(quasi)triangular $\Li$-bialgebroids} that generalizes the construction of Mackenzie--Xu. Here we do not elaborate that, but we hope to do so.

\medskip
\emph{Note about terminology and conventions.}  When we speak about vector spaces and algebras we  assume that they are $\Z$-graded, e.g., $V=V_0\oplus V_1$\,. A $\Z$-grading (\emph{parity}), which is always assumed, should be distinguished from a $\ZZ$-grading, which may be present or not. According to our  philosophy, a  $\ZZ$-grading  is an extra structure which comes e.g. from  a linear structure (such as that of a vector space or vector bundle) and should be seen as its replacement when it is absent.
When both kinds of gradings are present, they are assumed to be independent. Parity of an element is denoted by the tilde over the element's symbol.

\section{Recollection of homotopy  algebras and various bracket structures} \label{sec.recall}

In this section we recall various notions related with bracket structures that we shall need. The reader who believes he knows all that can skip it and go directly to the next section. However, there is one idea that we wish to stress here: different manifestations of an algebraic structure such as a Lie (super)algebra, Lie algebroid, $\Li$-algebra and $\Li$-algebroid.

\subsection{Lie superalgebras. Poisson and Schouten brackets. Canonical examples. ``Master Hamiltonians''}

It would help to recall first the notion of a Lie superalgebra and by using it as a toy example to explain the approach that we   apply to the more complicated case of $\Li$-algebras and $\Li$-algebroids. Thus,  a \emph{Lie superalgebra} $L$ is a vector space $L=L_0\oplus L_1$ endowed with  a bracket with the following properties: it is even, $[L_i,L_j]\subset L_{i+j}$;   bilinear,   in particular,
\begin{equation}\label{eq.liebilin}
    [\la u,v]=\la [u,v],  \quad [u,v\la]= [u,v]\la\,,
\end{equation}
for a scalar $\la$ of arbitrary parity;  antisymmetric:
\begin{equation}\label{eq.lieantisym}
    [u,v]=-(-1)^{\ut\vt}[v,u]\,;
\end{equation}
and satisfies the Jacobi identity:
\begin{equation}\label{eq.liejac}
    [u,[v,w]]=[[u,v],w]+(-1)^{\ut\vt}[v,[u,w]]\,.
\end{equation}
By using the antisymmetry property the Jacobi identity can be also re-written in the ``more classical'' form as
\begin{equation}\label{eq.liejacclass}
    (-1)^{\ut\wt}[u,[v,w]]+(-1)^{\wt\vt}[w,[u,v]]+(-1)^{\vt\ut}[v,[w,u]]=0\,.
\end{equation}
A \emph{morphism} (or \emph{homomorphism}) of Lie superalgebras is an even linear map $\f\co L\to K$ (so in particular $\f(\la u)=\la \f(u)$) that preserves the brackets:
\begin{equation}\label{eq.liehom}
    \f([u,v])=[\f(u),\f(v)]\,.
\end{equation}

Our first remark is that though we speak about `vector spaces', we always assume the possibility of extending scalars and introducing some commutative superalgebra as the ground ring. This explains our careful mentioning of the linearity properties, which always allow for `odd constants'. This helps when we move on to vector bundles over supermanifolds and algebroids. Another remark is that when required, we consider our vector spaces as supermanifolds, i.e., do not distinguish a ($\Z$-graded) vector space $V$ from the canonically associated with it `vector supermanifold', which we usually (with some exceptions)     denote by the same symbol.  A `vector supermanifold' has a preferred atlas of global charts in which all changes of coordinates are linear homogeneous.  An even linear map of vector spaces gives a supermanifold map, linear relative the preferred atlases, which we shall denote by the same symbol.

For a Lie superalgebra $L$, consider an   algebraic structure induced on the parity-reversed vector space $\Pi L$. If the definition of the bracket is not changed, only the parities of the elements are reversed, then   all the identities above stay  with the   parities shifted by $1$.
It is formally convenient    to combine the two cases and define a \emph{Lie superalgebra of parity} $\e\in\Z $ as a vector space $L$ with a bracket operation of parity $\e$, i.e.
\begin{equation}\label{eq.lieeps}
      [L_i,L_j]\subset L_{i+j+\e}\,, 
\end{equation}
which is bilinear:
\begin{equation}\label{eq.liebilineps}
    [\la u,v]=\la [u,v],  \quad [u,v\la]= [u,v]\la\,,
\end{equation}
(note no   signs!),  antisymmetric:
\begin{equation}\label{eq.lieantisymeps}
    [u,v]=-(-1)^{(\ut+1)(\vt+\e)}[v,u]\,,
\end{equation}
and satisfies the Jacobi identity:
\begin{equation}\label{eq.liejaceps}
    [u,[v,w]]=[[u,v],w]+(-1)^{(\ut+\e)(\vt+\e)}[v,[u,w]]\,.
\end{equation}
For $\e=0$, we come back to the standard notion; we may refer to it as an `even' Lie superalgebra. For $\e=1$, we speak about an \emph{odd Lie superalgebra}. If $L$ is a Lie superalgebra of parity $\e$, the the formula
\begin{equation}\label{eq.liepi}
    [\Pi u, \Pi v]:= \Pi [u,v]
\end{equation}
(no  extra signs!) makes $\Pi L$ a Lie superalgebra of parity $\e+1$. Note a peculiar fact: as a consequence of the bilinearity  holding in the form~\eqref{eq.liebilineps}, there is an identity 
\begin{equation}\label{eq.oddcomma}
    [u\la,v] =(-1)^{\lat\e}[u,\la v]\,,
\end{equation}
meaning that from the viewpoint of the sign rule, the parity of the bracket  in a Lie superalgebra of parity $\e$ `sits' on the central comma, which for $\e=1$ has to be regarded as an odd symbol. Later we shall see that there is another sign convention for a bracket in $\Pi L$ which can be more convenient in some problems. 

Now we turn to different versions of Poisson brackets. In a Poisson algebra (we shall recall formal definitions in a moment), there are two operations, bracket and associative product, so their relative parity cannot be altered by any parity shifts. We shall always assume that the associative product is even.

\begin{definition} A \emph{Poisson algebra of parity} $\e\in \Z$ is a vector space $A=A_0\oplus A_1$ with 
an even associative bilinear multiplication,
\begin{equation}\label{eq.ass}
      A_i A_j \subset A_{i+j}\,,
\end{equation}
and a bracket of parity $\e$,
\begin{equation}\label{eq.poisseps}
      \{A_i,A_j\}\subset A_{i+j+\e}\,,
\end{equation} 
with respect to which $A$ is a Lie superalgebra of parity $\e$, and there is the Leibniz identity:
\begin{equation}\label{eq.leibpoiss}
    \{a, bc\}=\{a,b\}c +(-1)^{(\at+1)\bt} b\{a,c\}\,.
\end{equation}
\end{definition}
One can check  using  the antisymmetry   that~\eqref{eq.leibpoiss} can be also re-written as the derivation property `from the right':
\begin{equation}\label{eq.leibright}
    \{ab, c\}=a\{b,c\} + (-1)^{\bt(\ct+\e)}\{a,c\}b \,,
\end{equation}
which is sometimes useful.
The bracket is called \emph{Poisson bracket}. 
A Poisson algebra  of parity $0$ is referred to simply as a \emph{Poisson algebra} or   \emph{even Poisson algebra}  when it is necessary to stress the parity; a Poisson algebra of parity $1$ is also called \emph{odd Poisson algebra} or   \emph{Schouten algebra} or   \emph{Gerstenhaber algebra} (due to the prototypal examples). The same terminology   applies to brackets. An odd Poisson bracket is   called   \emph{antibracket} in the physics literature and also a \emph{Buttin bracket} (after Claudette Buttin, see~\cite{buttin:schouten,buttin:last}).
We use above curly brackets, but in concrete examples all kinds of bracket shapes are used. 

In examples the multiplication is often commutative, $ab=(-1)^{\at\bt}ba$, but we do not assume that so not to exclude e.g. an associative noncommutative algebra (such as the algebra of linear operators) with $\{A,B\}:=[A,B]=AB-(-1)^{\At\Bt}BA$. (See also~\cite{tv:poiss}.)

\begin{remark}
We work everywhere with $\Z$-grading (parity). A $\ZZ$-grading (`degree' or `weight') may appear in concrete examples, sometimes as independent from parity and sometimes so that parity coincides with   $\ZZ$-grading modulo $2$. We consider a $\ZZ$-grading when it naturally arises in geometrical examples, but do not assume that all our supermanifolds are \emph{a priori} graded. (See   the general notion of graded manifolds, i.e. supermanifolds with an extra $\ZZ$-grading in the structure sheaf,  in~\cite{tv:graded}.) Poisson structures of a non-zero weight and their generalizations became recently a prominent object of study under the name ``shifted Poisson structures'', see e.g. Pridham~\cite{pridham:outline2018}.
\end{remark}

Our main examples of Poisson algebras will be various Poisson algebras of functions, so we now concentrate on those. The prototypal examples of even and odd Poisson algebras are given by the two parallel examples of `canonical brackets' that we briefly recall.  (`Canonical' means that their definition does not require any extra structure.) Let $M$ be a supermanifold.
\begin{example}\label{ex.canpoi}
On functions on the cotangent bundle $T^*M$,   the (even) \emph{canonical Poisson bracket} 
\begin{equation}\label{eq.canpoiss}
    (H,G)= (-1)^{\Ht\at+\at}\der{H}{p_a}\der{G}{x^a} - (-1)^{\Ht\at}\der{H}{x^a}\der{G}{p_a}\,.
\end{equation}
In particular,
\begin{equation}\label{eq.canpoisscoord}
    (p_a,x^b)=\d_a^b=-(-1)^{\at} (x^b,p_a)\,.
\end{equation}
It is uniquely defined by the axioms of an even Poisson bracket together with the `initial conditions'  
\begin{equation}\label{eq.initcanpoiss}
    (f,g)=0\,, \quad (H_X,f)=X(f)\,, \quad (H_X,H_Y)=H_{[X,Y]}\,,
\end{equation}
where $f,g\in \fun(M)$, $X,Y\in \Vect(M)$, and $H_X:=X^a(x)p_a$ for $X=X^a(x)\lder{}{x^a}$.\,\footnote{\,$H$ in $H_X$ is for `Hamiltonian'.} Here $x^a$ are local coordinates on $M$ and $p_a$ are the corresponding fiber coordinates with the transformation law $p_a=\lder{x^{a'}}{x^a}\,p_{a'}$ (the same as for partial derivatives). These brackets can be also  obtained from the canonical even symplectic form $\o=d(dx^ap_a)=dp_adx^a$.
\end{example}

\begin{example}\label{ex.cansch}
On functions on the anticotangent bundle $\Pi T^*M$,   the (odd) \emph{canonical Schouten bracket} 
\begin{equation}\label{eq.cansch}
    \lsch P,R\rsch= (-1)^{(\Pt+1)(\at+1)}\der{P}{x^*_a}\der{R}{x^a} - (-1)^{(\Pt+1)\at}\der{P}{x^a}\der{R}{x^*_a}\,.
\end{equation}
In particular,
\begin{equation}\label{eq.canschcoord}
    \lsch x^*_a,x^b\rsch =\d_a^b=- \lsch x^b,x^*_a\rsch\,.
\end{equation}
It is uniquely defined by the axioms of an odd Poisson bracket together with the `initial conditions' 
\begin{equation}\label{eq.initcansch}
    \lsch f,g\rsch=0\,, \quad \lsch \mul_X,f\rsch=(-1)^{\Xt}X(f)\,, \quad \lsch \mul_X,\mul_Y\rsch=\mul_{[X,Y]}\,,
\end{equation}
where $f,g\in \fun(M)$, $X,Y\in \Vect(M)$, and $\mul_X:=(-1)^{\Xt}X^a(x)x^*_a$ for $X=X^a(x)\lder{}{x^a}$.\,\footnote{\,$P$ in $P_X$ is for ``polyvector'' (the Russian for  multivector).} Here   $x^*_a$ are   fiber coordinates in $\Pi T^*M$ (``antimomenta'') with the same transformation law as $p_a$, $x^*_a=\lder{x^{a'}}{x^a}\,x^*_{a'}$ but of the opposite parity $\tilde x^*_a=\at+1$. These brackets can be also  obtained from the canonical odd symplectic form $\o=d(dx^ax^*_a)=-(-1)^{\at}dx^*_adx^a$.
\end{example}

\begin{remark}
The choices of signs in the definitions of $H_X$ and $\mul_X$   and in the `initial conditions'~\eqref{eq.initcanpoiss} and~\eqref{eq.initcansch}  for the Poisson and Schouten brackets are basically unique and fixed by the requirement  of linearity.
\end{remark}

\begin{remark}
Note that for an ordinary manifold $M$, functions on $\Pi T^*M$ (which will be automatically polynomial in odd fiber coordinates) can be identified with inhomogeneous  {multivector fields} on $M$; and,  likewise, functions on $\Pi TM$ can be identified with inhomogeneous  differential forms on $M$. 
For simplicity, we shall apply the same terminology in the general super case.  We shall use the notations
\begin{equation}\label{eq.forms}
    \O(M):= \fun(\Pi TM)
\end{equation}
and 
\begin{equation}\label{eq.multf}
    \Mult(M):= \fun(\Pi T^*M)
\end{equation}
and with an abuse of language call the elements of these algebras,  \emph{differential forms} and \emph{multivector fields} on $M$, respectively. (Strictly speaking, for a supermanifold, the correct terms   should be `pseudodifferential forms'  and `pseudomultivector fields'.)
\end{remark}

If a Poisson bracket of parity $\e$ is defined in the algebra $\fun(M)$, then a supermanifold $M$ is called a \emph{Poisson manifold} (for $\e=0$) and \emph{odd Poisson} or \emph{Schouten manifold} (for $\e=1$). (We suppress the prefix `super-'.) The bracket is referred to as a `bracket on $M$' or a `Poisson structure  on $M$'. One can easily deduce the following.
\begin{proposition}
Let $x^a$ be local coordinates on $M$. For arbitrary functions $f$ and $g$,
\begin{equation}
    \{f,g\} = (-1)^{\ft\at+\at} \der{f}{x^a}\{x^a,x^b\}\der{g}{x^b}\,.
\end{equation}
regardless of the parity of the bracket 
$\e$. 
\end{proposition}
Consider the properties of $\{x^a,x^b\}$ in the even and odd cases separately.

For an even bracket ($\e=0$), we have  the antisymmetry $\{x^b,x^a\}= - (-1)^{\at\bt}\{x^a,x^b\}$. Set
\begin{equation}\label{eq.poistens}
    P^{ab}:=(-1)^{\at}\{x^a,x^b\}\,.
\end{equation}
This turns antisymmetry into symmetry with respect to reversed parity:
\begin{equation*}
    P^{ab} =(-1)^{(\at+1)(\bt+1)}P^{ba}\,.
\end{equation*}
A tensor with such property corresponds to an \textbf{even} fiberwise quadratic function $P\in \fun(\Pi T^*M)$,
\begin{equation}\label{eq.poistens2}
    P :=\frac{1}{2}\,P^{ab}x^*_bx^*_a\,.
\end{equation}

Similarly, for an odd Poisson bracket on $M$, we have the antisymmetry with respect to reversed parity $\{x^b,x^a\}= - (-1)^{(\at+1)(\bt+1)}\{x^a,x^b\}$. Set likewise
\begin{equation}\label{eq.schtens}
    H^{ab}:=(-1)^{\at}\{x^a,x^b\}\,.
\end{equation}
This gives a symmetric tensor:
\begin{equation*}
    H^{ab} =(-1)^{\at \bt}H^{ba}\,, 
\end{equation*}
and we obtain an \textbf{odd} fiberwise quadratic function $H\in \fun(T^*M)$,
\begin{equation}\label{eq.schtens2}
    H:=\frac{1}{2}\,H^{ab}p_bp_a\,.
\end{equation}

The following theorem gives a coordinate-free relation between a Poisson bracket of parity $\e$ on a supermanifold $M$ and the objects $P$, $H$.
\begin{theorem}
An even Poisson bracket on $M$ is defined by an even fiberwise quadratic function $P\in\fun(\Pi T^*M)$ via the canonical Schouten bracket, by the formula
\begin{equation}\label{eq.poissderived}
    \{f,g\}:= \lsch f, \lsch P,g\rsch \rsch = \lsch \lsch f, P\rsch, g \rsch\,,
\end{equation}
with coordinate expressions given by~\eqref{eq.poistens},\eqref{eq.poistens2}. The Jacobi identity for $\{-,-\}$ is equivalent to the equation 
\begin{equation}\label{eq.pp}
    \lsch P,P \rsch=0 
\end{equation}
on the even function $P$.

An odd Poisson bracket on $M$ is defined by an odd fiberwise quadratic function $H\in\fun(T^*M)$ via the canonical Poisson bracket, by the formula
\begin{equation}\label{eq.schderived}
    \{f,g\}:= -(f, (H,g))  = -((f, H),g)\,,
\end{equation}
with coordinate expressions given by~\eqref{eq.schtens},\eqref{eq.schtens2}. The Jacobi identity for $\{-,-\}$ is equivalent to the equation
\begin{equation}\label{eq.hh}
    (P,P)=0 
\end{equation}
on the odd function  $H$.
\end{theorem}
(The minus sign in~\eqref{eq.schderived} is a result of   conventions. One can get rid of it. See later.)

We may write $\{-,-\}_{P,H}$  to emphasize that a bracket on $M$ is defined by $P$ or $H$.

Formulas~\eqref{eq.poissderived} and~\eqref{eq.schderived} are particular instances of derived bracket  introduced in~\cite{yvette:derived} (see also exposition in~\cite{yvette:derived-survey2004} and~\cite{tv:graded}).

An odd function $H\in \fun(T^*M)$  defining an odd Poisson bracket on $M$ will be called its \emph{master Hamiltonian}. (It satisfies the  \textbf{master equation}~\eqref{eq.hh}.) With some abuse of language, we shall apply the same terminology to an even function $P\in \fun(\Pi T^*M)$ defining an even Poisson bracket on $M$ (a  Poisson tensor  or    bivector in the classical terminology). (If one wishes to be consistent with the `anti-' terminology for $\Pi T^*M$, a good name for $P$ would be an ``antihamiltonian''.)  

For better visualization of the two parallel constructions, we can bring them together as a table\footnote{Borrowed with some modification from~\cite[\S 1]{tv:graded}. In particular, our $H^{ab}$ differs from $S^{ab}$ in~\cite{tv:graded}  by the minus sign.}:
\begin{center}
{\renewcommand{\arraystretch}{1.3}
\begin{tabular}{|c|c|} \hline
  \rule{0pt}{18pt} An even Poisson manifold $M$ & An odd Poisson or Schouten manifold $M$  \\
  \hline
  \multicolumn{2}{|c|}{\rule{0pt}{21pt}$\displaystyle
  \{f,g\}_{P,H}= (-1)^{\at(\ft+1)}\der{f}{x^a}\{x^a,x^b\}_{P,H}\der{g}{x^b}$}\\
  \multicolumn{2}{|c|}{\small(same coordinate formula for even and odd bracket)}\\
  \hline
  $\{x^a,x^b\}_P=(-1)^{\at}\,P^{ab}$,  & $\{x^a,x^b\}_H= (-1)^{\at}\,H^{ab}$,   \\
$P^{ab}$ is  `Poisson tensor'  & $H^{ab}$  is `Schouten tensor' \\
$P^{ab}=(-1)^{(\at+1)(\bt+1)}P^{ba}$ & $H^{ab}=(-1)^{\at\bt}H^{ba}$\\
$\widetilde{P^{ab}}=\at+\bt$ & $\widetilde{H^{ab}}=\at+\bt+1$\\
Master (anti)Hamiltonian: & Master Hamiltonian:
\\
$\displaystyle
P=\frac{1}{2}P^{ab}(x)x^*_bx^*_a\in \fun(\Pi T^*M)$ & $\displaystyle
H=\frac{1}{2}H^{ab}(x)p_bp_a\in \fun(T^*M)$\\
Jacobi for $\{\ ,\ \}_P$ $\Leftrightarrow$ $\lsch P,P\rsch=0$ & Jacobi for $\{\ ,\ \}_H$
$\Leftrightarrow$ $(H,H)=0$\\
Explicit formula: & Explicit formula:\\
$\displaystyle \{f,g\}_P=\lsch f,\lsch P,g\rsch\rsch$ & $\displaystyle
\{f,g\}_H=-(f,(H,g))$\\
\rule{0pt}{22pt}\parbox{6.5cm}{ \small Here $\lsch -,-\rsch$ is the canonical  Schouten bracket on $\Pi %
T^*M$. $P$ is even.}  & \rule{0pt}{22pt}\parbox{6.5cm}{\small Here $(- ,-)$ is the canonical
Poisson bracket on $T^*M$.
\small  $H$ is odd.} \\
\rule{0pt}{-20pt} &  \rule{0pt}{-20pt}\\
   \hline
\end{tabular}
}
\end{center}

\medskip
The canonical brackets on $T^*M$ and $\Pi T^*M$ have their own ``master Hamiltonians'', which are interesting to see. 

\begin{example}[Master Hamiltonian for the canonical Schouten bracket]
Let $x^a,x^*_a,p_a,\pi^a$ be coordinates on $T^*(\Pi T^*M)$, where $p_a$ and $\pi^a$ are the conjugate momenta for $x^a$ and $x^*_a$. (Note that $\tilde \pi^a=\tilde x^*_a=\at+1$.) One can check that the following odd function on $T^*(\Pi T^*M)$ is invariant:
\begin{equation}\label{eq.hamforsch}
    \Sch:=(-1)^{\at} \pi^ap_a\,,
\end{equation}
and that
\begin{equation}\label{eq.hamforsch2}
    ((P,\Sch),R)=\lsch P,R\rsch\,,
\end{equation}
for arbitrary $P,R\in \fun(\Pi T^*M)$. Hence $\Sch$ is up to the minus sign the master Hamiltonian for the canonical Schouten bracket.
\end{example}

\begin{example}[Master Hamiltonian for the canonical Poisson bracket]
Similarly, in canonical coordinates $(x^a,p_a,x^*_a,p^{*a})$    on $\Pi T^*(T^*M)$, one can check that the even function
\begin{equation}\label{eq.hamforpoi}
    \Poi:=  p^{*a}x^*_a\,,
\end{equation}
is invariant and satisfies
\begin{equation}\label{eq.hamfopoi2}
    \lsch\lsch H,\Poi\rsch,G\rsch=(H,G)\,,
\end{equation}
for arbitrary $H,G\in \fun(T^*M)$. 
\end{example}

\begin{remark} Both $\Sch$ and $\Poi$ are obtained from the `linear Hamiltonians' $H_d=dx^ap_a$ and $P_d=-dx^ax^*_a$ corresponding to the de Rham differential $d=dx^a\lder{}{x^a}$ (in the notation of Examples~\ref{ex.canpoi} and~\ref{ex.cansch})
by applying the Mackenzie--Xu diffeomorphism~\cite{mackenzie:bialg} and its odd analog~\cite{tv:graded}, $T^*(\Pi TM)\cong T^*(\Pi T^*M)$ and $\Pi T^*(\Pi TM)\cong \Pi T^*(T^*M)$. See also~\cite[\S 2]{tv:microformal}. It is a helpful exercise to check directly that $(-1)^{\at} \pi^a$ and $p^{*a}$ transform as $dx^a$.
\end{remark}

\subsection{Lie algebroids. Manifestations of Lie  brackets}
On functions on the dual space to a Lie algebra there is a Poisson bracket called Lie--Poisson or Berezin--Kirillov bracket. It was introduced by S.~Lie and rediscovered in the form a symplectic structure on coadjoint orbits by Kirillov, Kostant and Souriau and as a Poisson bracket by Berezin. It was Weinstein who observed that it had  already been known to Lie. The same is true for Lie superalgebras. There is also a parallel Schouten bracket on functions on the antidual (reversed parity dual) space. Our point is that all these structures and one more  that we shall now recall, should be seen as equivalent manifestations of a Lie bracket. We shall extend this picture to Lie algebroids and $\Li$-algebras. 

\begin{definition}
A  vector field $Q$ on a supermanifold $M$ is called \emph{homological} if   $Q$ is odd and $[Q,Q]=2Q^2=0$. A supermanifold endowed with a homological vector field $Q$ is called a \emph{$Q$-manifold}; the field $Q$ is also referred to as a \emph{$Q$-structure}. A \emph{morphism} of $Q$-manifolds (or a \emph{$Q$-morphism} or a \emph{$Q$-map}) is a smooth map $\f\co M_1\to M_2$  that intertwines the $Q$-structures on $M_1$ and $M_2$ (i.e. $Q_1$ and $Q_2$ on $M_1$ and $M_2$ are $\f$-related).
\end{definition}

Recall that arbitrary vector fields $X_1$ and $X_2$ on (super)manifolds $M_1$ and $M_2$ are $\f$-related for a map $\f\co M_1\to M_2$ if (equivalently) 
\begin{equation*}
   \f^* \circ X_2 = X_1\circ \f^*
\end{equation*}
as operators on functions or the diagram
\begin{equation*} 
    \begin{CD} (\Pi)TM_1   @>{T\f}>> (\Pi)TM_2 \\
                @A{X_1}AA         @AA{X_2}A\\
               M_1@>{\f}>> M_2 \,,
    \end{CD}
\end{equation*}
commutes. (For odd fields one has to consider the bundles $\Pi TM_1$ and $\Pi TM_2$.) In local coordinates this is expressed by
\begin{equation}\label{eq.relatedvfs}
    X_1^a(x) \der{\f^i}{x^a} =X_2^i(\f(x))\,.
\end{equation}

Consider a vector space $L=L_0\oplus L_1$ (which we when necessary consider as a supermanifold). The linear structure induces in the algebra of functions a  non-negative $\ZZ$-grading  by the degrees of linear coordinates.   We shall refer to it as \emph{weight}. Parity and weight are a priori independent. Consider all   neighbors of $L$, i.e. the vector spaces $L^*$, $\Pi L$ and $\Pi L^*$.

\begin{proposition}[Equivalent manifestations of a Lie superalgebra]
A  Lie superalgebra structure in   $L$ is  equivalent to the following algebraic and geometric structures on its neighbors:
\begin{itemize}
  \item An odd Lie superalgebra structure in a vector space $\Pi L$;
  \item A homological vector field of weight $+1$ on the supermanifold $\Pi L$;
  \item An even Poisson bracket of weight $-1$ on the supermanifold $L^*$;
  \item An odd Poisson bracket of weight $-1$ on the supermanifold $\Pi L^*$. 
\end{itemize}
\end{proposition} 
The   brackets on $L^*$ and $\Pi L^*$ are the \emph{Lie--Poisson} and \emph{Lie--Schouten brackets}, respectively. A $Q$-structure of weight $+1$ is give by an odd vector field on $\Pi L$,
\begin{equation*}
    Q=\frac{1}{2}\,\x^i\x^jQ_{ji}^k\der{}{\x^k}\,.
\end{equation*}
It defines a Lie bracket of elements of $L$ by the formula
\begin{equation*}
    \io_{[u,v]} =(-1)^{\ut}[[Q,\io_u],\io_v]
\end{equation*}
where $\io_u=(-1)^{\ut} u^i\lder{}{\x^i}$ for $u=u^ie_i\in L$. The Jacobi identity for the bracket in $L$ is equivalent to $Q^2=0$. For the basis elements we obtain
\begin{equation*}
    [e_i,e_j]= (-1)^{\jt}Q_{ij}^ke_k\,.
\end{equation*}
The Lie--Poisson and Lie--Schouten brackets are defined, respectively, by
\begin{equation*}
    \{x_i,x_j\}= (-1)^{\jt}Q_{ij}^kx_k\,.
\end{equation*}
and
\begin{equation*}
    \{\h_i,\h_j\}= (-1)^{\itt}Q_{ij}^k\h_k\,.
\end{equation*}

In this language,  Lie algebra homomorphisms are  described as follows.
\begin{proposition}[Lie  superalgebra homomorphism in different manifestations]
A homomorphism of  Lie superalgebras   $\f\co L\to K$ in terms of structures on neighbors  is  equivalent to : 
\begin{itemize}
  \item A homomorphism of  odd Lie superalgebras $\Pi L\to \Pi K$;
  \item A linear  $Q$-map $\Pi L\to \Pi K$;
  \item A linear Poisson map   $K^*\to L^*$ for the Lie--Poisson brackets;
  \item A linear Poisson map   $\Pi K^*\to \Pi L^*$ for the Lie--Schouten brackets.
\end{itemize}
\end{proposition} 

Recall that a Poisson map between Poisson manifolds $M_1$ and $M_2$ (with Poisson brackets of the same parity) is defined as a smooth map $F\co M_1\to M_2$ such that the corresponding Poisson tensors are $F$-related.  For contravariant tensor fields it is defined    similarly  to vector fields.  One can easily see that this is equivalent to the condition that the pullback of functions preserves Poisson brackets.
 
This picture carries through to Lie algebroids. Recall that a \emph{Lie algebroid} with base $M$ is a vector bundle $E\to M$ for which the space of sections $\fun(M,E)$ is endowed with a structure of a Lie superalgebra so that the bracket of sections is related with the multiplication by functions on the base by the Leibniz formula
\begin{equation}\label{eq.leibliealg}
    [u,fv]=a(u)(f)\,v +(-1)^{\ut\ft}f[u,v]
\end{equation}
where $a\co E\to TM$ is a vector bundle morphism called \emph{anchor}. (See e.g.~\cite{mackenzie:book2005}. We define Lie algebroids in the super setting, but it makes no difference.) We can consider the vector bundles over $M$ which are the neighbors of $E$. 

\begin{proposition}[Equivalent manifestations of a Lie algebroid]
A  Lie algebroid structure in   $E$ is  equivalent to the following algebraic and geometric structures on its neighbors:
\begin{itemize}
  \item An odd Lie algebroid structure in a vector space $\Pi E$;
  \item A homological vector field of weight $+1$ on the supermanifold $\Pi E$;
  \item An even Poisson bracket of weight $-1$ on the supermanifold $E^*$;
  \item An odd Poisson bracket of weight $-1$ on the supermanifold $\Pi E^*$.
\end{itemize}
\end{proposition} 
(We skip the definition of an `odd Lie algebroid'. It should be clear from the statement.) 
`Weight' for functions on vector bundles is defined as the degree in fiber coordinates. 
The characterization of Lie algebroids in terms of homological vector fields is due to Vaintrob~\cite{vaintrob:algebroids}. In local coordinates,
\begin{equation*}
    Q=\x^iQ_i^a\der{}{x^a}+ \frac{1}{2}\,\x^i\x^jQ_{ji}^k\der{}{\x^k}\,.
\end{equation*}
Similarly to the above, a map $u\mapsto \io_u$ defines an odd isomorphism of sections of $E$ and vector fields on $\Pi L$ of weight $-1$.
The brackets and the anchor  on sections are given by
\begin{equation*}
    \io_{[u,v]} =(-1)^{\ut}[[Q,\io_u],\io_v]\,, \quad a(u)(f)=[Q,\io_u](f)\,.
\end{equation*}
We skip the formulas for the Lie--Poisson and Lie--Schouten brackets on $E^*$ and $\Pi E^*$ (see e.g.~\cite{tv:graded}\cite{tv:qman-mack}) because below we consider a more general (homotopy) case.  
 
Now we would like to see how morphisms of Lie algebroids look like in different manifestations. Here is some subtlety. For Lie algebroids over the same base, there is an obvious notion of a morphism, which is a fiberwise map over  fixed base preserving brackets and anchors. Higgins and Mackenzie defined a morphism of Lie algebroids over different bases~\cite{mackenzie_and_higgins:algebraic}, see~\cite[\S 4.3]{mackenzie:book2005}, as a morphism of the underlying vector bundles satisfying certain conditions, which are non-obvious because in this case there is no   mapping of sections.  As it has turned out, the simplest description is via $Q$-structures: a vector bundle morphism 
\begin{equation*}
    \begin{CD} E_1@>{\bar\f}>> E_2\\
                @VVV         @VVV\\
                M_1@>{\f}>> M_2 \,,
    \end{CD}
\end{equation*}
is a \emph{Lie algebroid morphism} if and only if the induced map $\bar \f^{\Pi}$ of the parity-reversed bundles 
\begin{equation*}
    \begin{CD} \Pi E_1@>{\bar\f^{\Pi}}>> \Pi E_2\\
                @VVV         @VVV\\
                M_1@>{\f}>> M_2 \,,
    \end{CD}
\end{equation*}
is a $Q$-morphism.
However, to avoid complications with dualization, we shall restrict ourselves here to morphisms over fixed base.\footnote{\,A morphism of vector bundles over different bases induces on the dual bundles a ``comorphism''. See e.g.~\cite{mackenzie_and_higgins:duality}.  There is also a notion of a Lie algebroid comorphism (ibid., also~\cite[\S 4.3]{mackenzie:book2005}),   parallel with the notion of morphism and  coinciding with it  over   a fixed base. To simplify our task, we do not go into that direction.} The description in terms of $Q$-structures remains valid, but there are also clear descriptions for two other manifestations. The statement is completely similar to the case of Lie (super)algebras.

\begin{proposition}[Lie  algebroid  morphism in different manifestations]
For Lie algebroids over the same base, the following are equivalent:
\begin{itemize}
  \item A Lie  algebroid  morphism (over fixed base), i.e. a vector bundle morphism with fixed base $E_1\to E_2$ preserving brackets of sections and anchors;
  \item A fiberwise linear  $Q$-map $\Pi E_1\to \Pi E_2$;
  \item A fiberwise linear Poisson map   $E_2^*\to E_1^*$ for the Lie--Poisson brackets;
  \item A fiberwise linear Poisson map   $\Pi E_2^*\to \Pi E_1^*$ for the Lie--Schouten brackets.
\end{itemize}
\end{proposition}

\subsection{$\Li$-algebras and $\Li$-algebroids. $\Pin$- and $\Si$-algebras.  Manifestations of $\Li$-structure}

Now we are going to recall brackets that have many arguments, not just binary as before. `Central comma' does not make sense for them, hence before moving forward we explain an alternative sign convention for   odd brackets. For concreteness, consider the case of $\Pi L$ for a Lie superalgebra $L$. Denote $V:=\Pi L$. If a bracket for $V=\Pi L$ is defined by
\begin{equation*}
    [\Pi u, \Pi v]:=\Pi [u,v]
\end{equation*}
as in~\eqref{eq.liepi}, it satisfies~\eqref{eq.liebilineps}, \eqref{eq.lieantisymeps} and \eqref{eq.liejaceps} with $\e=1$.  Define a bracket in $V$ slightly differently, by
\begin{equation}
    [\Pi u,\Pi v]':=(-1)^{\ut}\Pi [u,v]\,.
\end{equation}
(We temporarily use notation with the dash. Later we shall drop it.) The two versions are related by
\begin{equation}
    [\a,\b]':=(-1)^{\tilde \a+1}  [\a,\b]\,,
\end{equation}
for $\a,\b\in V$. One can see that this version of an odd bracket in $V$  has the following properties:
it is bilinear, in the form
\begin{equation}\label{eq.oliebilin}
    [\la \a,\b]'=(-1)^{\lat}\la [\a,\b]',  \quad [\a,\b\la]'= [\a,\b]'\la\,,
\end{equation}
symmetric:
\begin{equation}\label{eq.olieasym}
    [\a,\b]'=(-1)^{\att\btt}[\b,\a]'\,,
\end{equation}
and satisfies the following form of the Jacobi identity:
\begin{equation}\label{eq.oliejac}
    [\a,[\b,\g]]=(-1)^{\att+1}[[\a,\b],\g]+(-1)^{(\att+1)(\btt+1)}[\b,[\a,\g]]\,.
\end{equation}
Bilinearity in the form~\eqref{eq.oliebilin}  means that parity `sits' at the opening bracket, which should be regarded as an odd symbol. The Jacobi identity in the form~\eqref{eq.oliejac} is consistent with that. It can be re-written also in the cyclic form
\begin{equation}\label{eq.oddliejaccycl}
    (-1)^{\att(\gtt+1)}[\a,[\b,\g]']' + (-1)^{\gtt(\btt+1)}[\g,[\a,\b]']' + (-1)^{\btt(\att+1)}[\b,[\g,\a]']' =0
\end{equation}
and   as 
\begin{equation}
    [[\a,\b]',\g]'+(-1)^{\btt\gtt}[[\a,\g]',\b]'+(-1)^{\att(\btt+\gtt)+\btt\gtt}[[\g,\b]',\a]'=0\,,
\end{equation}
where  the sum over   $(2,1)$-shuffles. It is the latter form that conveniently extends to the  homotopy  case  which we  shall now  introduce.

Recall the definition of  $\Li$-algebras (also known as strongly homotopy Lie algebras or SHLAs) due to Lada and Stasheff~\cite{lada:stasheff}. (Examples   appeared in physics literature, notably Zwiebach~\cite{zwiebach:93csft}).   There are two equivalent versions that differ by parities and we need both. (We use $\Z$-grading, not $\ZZ$-grading,  as in Lada--Stasheff.)
\begin{definition}[\emph{$\Linf$-algebra:  antisymmetric version}]
An \emph{$\Li$-algebra} is a  vector space $L=L_0\oplus L_1$ endowed with a sequence of multilinear operations
\begin{equation*}
    [-,\ldots,-]\co \underbrace{L\times \ldots \times L}_{\text{$n$ times}} \to L \quad \text{(for $n=0,1,2,\ldots $)}
  \end{equation*}
such that
  \begin{itemize}
    \item the parity of the $n$th bracket is $n\mod 2$;
    \item all brackets are antisymmetric (in $\Z$-graded sense);
    \item `generalized Jacobi identities' with $n$ arguments hold for all~$n=0,1,2,3,\ldots$\,:
    \begin{equation}\label{eq.genjacanti}
        \sum_{r+s=n}(-1)^{rs}\sum_{\text{$(r,s)$-shuffles}}   (-1)^{\a} \sign \s [[u_{\s(1)},\ldots,u_{\s(r)}], \ldots , u_{\s(r+s)}]=0\,.
    \end{equation}
  \end{itemize}
Here $(-1)^{\a}=(-1)^{\a(\s; \ut_,\ldots,\ut_n)}$  is the Koszul sign, i.e. the sign arising from  a given permutation of commuting symbols of indicated parities. Multilinearity means
\begin{equation}\label{eq.nlinear}
    [u_1,\ldots,\la u_k,\ldots,u_n]=(-1)^{\lat(n+\ut_1+\ldots+\ut_{k-1})} \la [u_1,\ldots, u_k,\ldots,u_n]\,,
\end{equation}
and antisymmetry means  a flip of adjacent arguments  $u_k,u_{k+1}$ giving the sign $-(-1)^{\ut_k\ut_{k+1}}$.
\end{definition}

We shall refer to a bracket with $n$ arguments as an `$n$-bracket'. Note that the definition includes a $0$-bracket, which is a distinguished even element in $L$. (It is often assumed to be zero and when it is not zero, such an $\Li$-algebra is sometimes called `curved' or `with background'.)

\begin{example} Denote $[\void]=:\O$ (an even element of $L$) and $[u]=:Du$. The `lower' generalized Jacobi identities (for $n=0,1,2$) can be shown to be
\begin{align*}
    D\O&=0\,,\\
    D^2u&=-[\O,u]\,,\\
    D[u,v]&=[Du,v]+(-1)^{\ut}[u,Dv]-[\O,u,v]\,.
\end{align*}
\end{example}

The parallel version is as follows.
\begin{definition}[\emph{$\Linf$-algebra:  symmetric version}]
An \emph{$\Li$-algebra} (in symmetric version) is a  vector space $V=V_0\oplus V_1$ endowed with a sequence of multilinear operations
\begin{equation*}
    [-,\ldots,-]\co \underbrace{V\times \ldots \times V}_{\text{$n$ times}} \to V \quad \text{(for $n=0,1,2,\ldots $)}
  \end{equation*}
such that
  \begin{itemize}
    \item all brackets are odd;
    \item all brackets are symmetric (in $\Z$-graded sense);
    \item `generalized Jacobi identities' with $n$ arguments hold for all~$n=0,1,2,3,\ldots$ in the form 
    \begin{equation}\label{eq.genjacsym}
        \sum_{r+s=n} \sum_{\text{$(r,s)$-shuffles}}   (-1)^{\a}[[u_{\s(1)},\ldots,u_{\s(r)}], \ldots , u_{\s(r+s)}]=0\,.
    \end{equation}
  \end{itemize}
Here $(-1)^{\a}=(-1)^{\a(\s; \ut_,\ldots,\ut_n)}$  is the Koszul sign. Multilinearity means
\begin{equation}\label{eq.nlinearodd}
    [u_1,\ldots,\la u_k,\ldots,u_n]=(-1)^{\lat(1+\ut_1+\ldots+\ut_{k-1})} \la [u_1,\ldots, u_k,\ldots,u_n]\,,
\end{equation}
and  symmetry means  a flip of adjacent arguments $u_k,u_{k+1}$ giving the sign $+(-1)^{\ut_k\ut_{k+1}}$.
\end{definition}

For both versions, the opening bracket is the symbol that carries the parity. 
One can check that if $L$ is an $\Li$-algebra in the antisymmetric version, then the vector space $V=\Pi L$ will become  an $\Li$-algebra in the  symmetric version if the operations are defined by 
\begin{equation}\label{eq.relationbrack}
    [\Pi u_1,\ldots, \Pi u_n]= (-1)^{\e}\Pi [u_1,\ldots,u_n]\,, 
\end{equation}
where $\e=\sum \ut_k(n-k)$,  and vice versa. (E.g.   $[\Pi u,\Pi v]=(-1)^{\ut}\Pi [u,v]$.)

Now we recall $\Li$-algebroids (see e.g.~\cite{tv:higherpoisson}). As for $\Li$-algebras, there are two parallel versions, with antisymmetric or  symmetric brackets. To save space, we give here a definition only for the antisymmetric version (which directly generalizes ordinary Lie algebroids).

\begin{definition}
An  \emph{$\Li$-algebroid} over a base $M$ is a vector bundle $E\to M$ endowed with a sequence of brackets of parities $n$,
\begin{equation*}
    [-,\ldots,-]\co \underbrace{\fun(M,E)\times \ldots \times \fun(M,E)}_{\text{$n$ times}} \to \fun(M,E) 
  \end{equation*}
$n=0,1,2,\ldots $, where $\fun(M,E)$ is the space of sections,   making $\fun(M,E)$ an $\Li$-algebra (in the antisymmetric version), and a sequence of fiberwise multilinear antisymmetric maps
\begin{equation*}
    a_n(-,\ldots,-)\co  \underbrace{E\times_M \ldots \times_M E}_{\text{$n$ times}} \to TM\,,
  \end{equation*}
$n=0,1,2,\ldots $, of parities $n-1$, called \emph{(higher) anchors}, so that the brackets satisfy the following Leibniz identity with respect to multiplication by functions on $M$\,:
\begin{equation}\label{eq.highleib}
    [u_1,\ldots,u_{n}, fu_{n+1}]= a_n(u_1,\ldots,u_n)(f)u_{n+1} + (-1)^{\e} f [u_1,\ldots,u_{n}, u_{n+1}]\,,
\end{equation}
where $(-1)^{\e}=(-1)^{\ft(n+1+\ut_1+\ldots+\ut_{n})}$\,.
\end{definition}

(In particular, there is a $0$-anchor, which is a distinguished  odd vector field on $M$.)

The most efficient description of $\Li$-algebras and $\Li$-algebroids is again by homological vector fields.  

Consider a vector space $V=V_0\oplus V_1$ and a formal odd vector field $Q\in \Vect(V)$, 
\begin{equation}\label{eq.qexpansion}
    Q=Q_{-1}+Q_0+Q_1+Q_2+\ldots = \Bigl(Q^k + \x^iQ_i^k + \frac{1}{2}\,\x^i\x^jQ_{ji}^k + \frac{1}{3!} \x^i\x^j\x^l Q_{lji}^k +\ldots \Bigr)\der{}{\x^k}
\end{equation}
(subscript denotes weight), where $\x^i$ are linear coordinates on $V$.  Elements of both $V$ and $\Pi V$ can be identified with vector fields on $V$ of weight $-1$ by the maps $i\co V\to \Vect(V)$ and $\io\co \Pi V\to \Vect(V)$, $i(\a)=\a^i\lder{}{\x^i}$ if $\a=\a^i\e_i$ (we may simply identify $i(\a)=\a$) and   $\io(u)=(-1)^{\ut}u^i\lder{}{\x^i}$ if $u=u^ie_i$; here $e_i=\e_i\Pi$ and $\x^i$ are left coordinates, e.g. $\x^ie_i$ is invariant. (Note also  $i(\Pi u)=\io(u)$.) A vector field $Q$ defines brackets on $V$ and $\Pi V$ by
\begin{equation}\label{eq.brackoddlinf}
    [\a_1,\ldots,\a_n] =  [\ldots [Q, \a_1],\ldots, \a_n](0)
\end{equation}
and
\begin{equation}\label{eq.bracklinf}
    \io\bigl([u_1,\ldots,u_n]\bigr)=(-1)^{\e} [\ldots [Q,\io(u_1)],\ldots,\io(u_n)](0)\,,
\end{equation}
where $\e=\sum \ut_k(n-k)$. The brackets~\eqref{eq.brackoddlinf} are odd and symmetric, while the brackets~\eqref{eq.bracklinf} are antisymmetric and have alternating parities. They define   $\Linf$-algebra structures in $L=\Pi V$ (antisymmetric version) and $\Pi L=V$ (symmetric version) if and only if $Q^2=0$.

This immediately generalizes to $\Linf$-algebroids. If $E\to M$ is a vector bundle, consider the bundle $\Pi E\to M$ and vector fields on $\Pi E$. Again, the space of vector fields is $\ZZ$-graded,
\begin{equation*}
    \Vect(\Pi E)=\Vect_{-1}(\Pi E)\oplus \Vect_{0}(\Pi E)\oplus \Vect_{1}(\Pi E)\oplus \ldots
\end{equation*}
and there an odd isomorphism $\fun(M,E)\to \Vect_{-1}(\Pi E)$, 
\begin{equation*}
    u\mapsto \io(u)\,, \quad \text{where $\io(u)=(-1)^{\ut}u^i(x)\der{}{\x^i}$}
\end{equation*}
(same conventions about bases and coordinates as above). Let $Q$ be an odd formal vector field on $\Pi E$\,:
\begin{multline*}
    Q=Q_{-1}+Q_0+Q_1+ Q_2 + \ldots =  Q^k(x)\der{}{\x^k} + \left(Q^a(x)\der{}{x^a} + \x^iQ_i^k(x)\der{}{\x^k}\right) + \\
    \left(\x^iQ_i^a(x)\der{}{x^a} + \frac{1}{2}\,\x^i\x^jQ_{ji}^k(x)\der{}{\x^k}\right)+
    \left(\x^i\x^jQ_{ji}^a(x)\der{}{x^a} + \frac{1}{3!}\,\x^i\x^j\x^lQ_{lji}^k(x)\der{}{\x^k}\right)+
    \ldots
\end{multline*}
($x^a$ coordinates on the base, $\x^i$ coordinates in the fibers). We can define brackets and anchors for sections of $E$ as follows:
\begin{multline}\label{eq.bracklinfalgd}
    \io\bigl({[u_1,\ldots,u_n]}\bigr):= (-1)^{\e}[\ldots [Q,\io(u_1)],\ldots,\io(u_n)]_{-1}=\\
     (-1)^{\e}[\ldots [Q_{n-1},\io(u_1)],\ldots,\io(u_n)]\,,
\end{multline}
and
\begin{multline}\label{eq.anchlinfalgd}
   a_n(u_1,\ldots,u_n)(f):= (-1)^{\e} [\ldots [Q,\io(u_1)],\ldots,\io(u_n)]_{0}(f)=\\
   (-1)^{\e} [\ldots [Q_{n},\io(u_1)],\ldots,\io(u_n)]_{0}(f)\,,
\end{multline}
$f\in \fun(M)$. Here  the subscripts ${}_{-1}$ and ${}_{0}$ denote  the projections on the subspaces of   weights $-1$ and $0$, and the sign
 $(-1)^{\e}=(-1)^{\sum_k \ut_k(n-k)}$ is as above. One can check that these operations have the required parities and satisfy the multilinearity and antisymmetry conditions. It is useful to note that only one homogeneous component of $Q$ makes an input in each operation.

\begin{theorem} Operations~\eqref{eq.bracklinfalgd} and~\eqref{eq.anchlinfalgd} define in a vector bundle $E\to M$ a structure of an $\Linf$-algebroid over $M$ if and only if $Q^2=0$. 
\end{theorem} 
\begin{proof} In the direct sum decomposition
\begin{equation*}
    \Vect(\Pi E)= \Vect_{-1}(\Pi E)\oplus \Vect_{\geq 0}(\Pi E)
\end{equation*}
both   summands are Lie subalgebras and $\Vect_{-1}(\Pi E)$ is abelian.\footnote{\,Compare with general non-negatively graded $Q$-manifolds~\cite{tv:qman}.} Hence results of~\cite{tv:higherder} apply and the generalized Jacobi identities for the brackets~\eqref{eq.bracklinfalgd} hold if and only if $Q^2=0$. The fact that the Leibniz identity~\eqref{eq.highleib} is easily checked directly. One also need to check that~\eqref{eq.anchlinfalgd} is multilinear over $\fun(M)$ and thus defines a multilinear map of vector bundles. By antisymmetry, it is enough to check $a_n(u_1,\ldots,u_{n-1},fu_n)(g)=\pm [\ldots [[Q_n,\io(u_1)],\ldots,\io(u_{n-1}),\io(fu_n)](g)$, which has the form 
\begin{equation*}
   \pm [X_{+1},f\io(u_n)]g=\pm X_{+1}(f)\io(u_n)(g) \pm   f[X_{+1}, \io(u_n)](g)=\pm   f[X_{+1}, \io(u_n)](g)\,,
\end{equation*}
where the first term vanishes because vector fields of the form $\io(u)$ kill functions on $M$.
\end{proof}

We have finally reached homotopy analogs of even and odd Poisson algebras. 

The following is the analog of an even Poisson algebra.
\begin{definition} A $\Pinf$-algebra (or homotopy Poisson algebra) is a vector space $A=A_0\oplus A_1$ endowed with an even bilinear associative multiplication and a sequence of bracket operations
\begin{equation*}
    \{-,\ldots,-\}\co \underbrace{A\times \ldots \times A}_{\text{$n$ times}} \to A\,, 
  \end{equation*}
of parities $n \mod 2$,  $n=0,1,2,\ldots $, which define on $A$ an $\Linf$-algebra structure  in the antisymmetric version and which are multiderivations with respect to the associative product. 
\end{definition}

Similarly, the following is the analog of an odd Poisson  (=Schouten, Gerstenhaber) algebra.
\begin{definition} An $\Sinf$-algebra (or homotopy Schouten algebra) is a vector space $A=A_0\oplus A_1$ endowed with an even bilinear associative multiplication and a sequence of bracket operations
\begin{equation*}
    \{-,\ldots,-\}\co \underbrace{A\times \ldots \times A}_{\text{$n$ times}} \to A\,,
  \end{equation*}
$n=0,1,2,\ldots $, all odd,  which define on $A$ an $\Linf$-algebra structure  in the  symmetric version and which are multiderivations with respect to the associative product.
\end{definition}

When $A=\fun(M)$ for a supermanifold $M$, $\Pinf$- and $\Sinf$-structures on $M$ are described by ``master Hamiltonians'' (which unlike ordinary Poisson or Schouten structures do not have to be quadratic). 

Before giving the corresponding formulas, it will be convenient to re-define signs in the canonical Schouten bracket. If $\lsch P, R\rsch$ is the Schouten bracket in the algebra $\fun(\Pi T^*M)$ as defined in Example~\ref{ex.cansch}, set
\begin{equation}\label{eq.newcansch}
    [P,R]:=(-1)^{\Pt+1}\lsch P,R\rsch\,.
\end{equation}
(We shall still refer to it as Schouten bracket.) It now satisfies
\begin{equation*}
    [\la P,R]=(-1)^{\lat}\la [P,R]\,,
\end{equation*}
\begin{equation*}
    [P,R]=(-1)^{\Pt\Rt}[R,P]\,,
\end{equation*}
and 
\begin{equation*}
    [P,[R,Q]]=(-1)^{\Pt+1}[[P,R],Q]+ (-1)^{(\Pt+1)(\Rt+1)}[R,[P,Q]]\,,
\end{equation*}
The explicit coordinate formula will be
\begin{equation}\label{eq.newcanschexplic}
    [P,R]= (-1)^{\at(\Pt+1)}\der{P}{x^*_a}\der{R}{x^a} + (-1)^{\Pt+ \at(\Pt+1)}\der{P}{x^a}\der{R}{x^*_a}\,,
\end{equation}
in particular,
\begin{equation}\label{eq.new canschcoord}
    [x^*_a,x^b]=(-1)^{\at}\d_a^b=  [x^b,x^*_a]\,.
\end{equation}

A \emph{master (anti)Hamiltonian} for a $\Pinf$-structure on $M$ is an  even  formal function $P\in \fun(\Pi T^*M)$ (i.e. treated as a formal expansion around the zero section in $\Pi T^*M$) \,---\,we shall suppress the adjective usually,\,---\, $P=P(x,x^*)$,
\begin{equation}\label{eq.antimasterp}
    P=P_0+P_1+P_2+P_3+\ldots = P_0(x) + P^a(x)x^*_a +\frac{1}{2}P^{ab}x^*_bx^*_a +\frac{1}{3!}P^{abc}x^*_cx^*_bx^*_a +\ldots  
\end{equation}
satisfying the equation
\begin{equation}\label{eq.mastereqp}
    [P,P]=0\,.
\end{equation}
Brackets generated by $P$ are defined by the formula
\begin{equation}\label{eq.pinfder}
    \{f_1,\ldots,f_n\}_P:= [\ldots [P,f_1],\ldots,f_n]|_{M}\,.
\end{equation}
Here at the right-hand side $[-,-]$ denotes the canonical Schouten bracket re-defined as above. It is clear that the bracket given by~\eqref{eq.pinfder} has parity $n\mod 2$ and satisfies
\begin{equation*}
     \{\la f_1,\ldots,f_n\}_P=(-1)^{n\lat}\la \{f_1,\ldots,f_n\}_P\,,
\end{equation*}
so the opening bracket carries the parity. One can check that the brackets are antisymmetric in the $\Z$-graded sense. Finally, the generalized Jacobi identities are equivalent to the condition $[P,P]=0$ (by the results of~\cite{tv:higherder}).

In a similar way, a \emph{master  Hamiltonian} for an $\Sinf$-structure on $M$ is an  odd  formal function $H\in \fun(T^*M)$ , $H=H(x,p)$,
\begin{equation}\label{eq.masterp}
    H=H_0+H_1+H_2+H_3+\ldots = H_0(x) + H^a(x)p_a +\frac{1}{2}H^{ab}p_bp_a +\frac{1}{3!}H^{abc}p_cp_bp_a +\ldots
\end{equation}
satisfying the equation
\begin{equation}\label{eq.mastereqh}
    (H,H)=0 
\end{equation}
(with the canonical Poisson bracket). 
Brackets generated by $H$ are defined by the formula
\begin{equation}\label{eq.sinfder}
    \{f_1,\ldots,f_n\}_H:= (\ldots (H,f_1),\ldots,f_n)|_{M}\,.
\end{equation}
Here at the right-hand side $(-,-)$ denotes the canonical Poisson bracket on $T^*M$. Since $H$ is odd and the canonical Poisson bracket is even, all brackets~\eqref{eq.sinfder} are odd. Parity is carried by the opening bracket. 
One   checks that the brackets are symmetric in the $\Z$-graded sense. The generalized Jacobi identities are equivalent to the condition $(H,H)=0$.

Suppose now we have an $\Linf$-algebra $L$ or an $\Linf$-algebroid $E\to M$. With the notions of $\Pinf$- and $\Sinf$-structures on manifolds at hand, we are able to supplement the picture given above by providing analogs of the Lie--Poisson and Lie--Schouten bracket for Lie (super)algebras or Lie algebroids. In the $\Linf$-case it is convenient to start from a homological vector field. For concreteness, consider an $\Linf$-algebra $L$. To it corresponds a homological vector field $Q\in \Vect(\Pi L)$. We lift it to $T^*(\Pi L)$ and $\Pi T^*L$ as `linear Hamiltonians' $H_Q$ and $\mul_Q$ (see the formulas in Examples~\ref{ex.canpoi} and~\ref{ex.cansch}). This preserves the brackets, so we obtain Hamiltonians satisfying $(H_Q,H_Q)=0$ and $[\mul_Q,\mul_Q]=0$ (note that $H_Q$ is odd and $\mul_Q$ is even). After that we apply the Mackenzie--Xu canonical diffeomorphism $T^*(\Pi L)\cong T^*(\Pi L^*)$ and its odd analog $\Pi T^*(\Pi L)\cong \Pi T^*(L^*)$ (see~\cite{tv:graded} and \cite[\S 2]{tv:microformal}). This gives an odd Hamiltonian $H$ on $\Pi L^*$ satisfying $(H,H)=0$ and an even `antiHamiltonian' $P$ on $L^*$ satisfying $[P,P]=0$\,. By definition, we take them as master Hamiltonians for a sequence of odd \emph{homotopy Lie--Schouten brackets} on $\Pi L^*$ (which form an $\Sinf$-algebra) and a sequence of \emph{homotopy  Lie--Poisson brackets} of alternating parities on $L^*$ (which form a $\Pinf$-algebra).

Note that in the same way as classical Lie--Poisson and Lie--Schouten bracket are distinguished by being ``linear brackets'' (which is conveniently formulated as having weight $-1$), the homotopy brackets obtained are also distinguished by their weights. One can check that (for both cases), a bracket of arity $n$ has weight $1-n$. (In the classical case of $n=2$, $1-2=-1$.) For example, if $Q$ on $\Pi L$ is given by~\eqref{eq.qexpansion}, then (up to signs, that can be found)
\begin{equation*}
    \{\h_{i_1}, \ldots,\h_{i_n}\}= \pm Q^k_{i_1\ldots i_n} \h_k
\end{equation*}
for the homotopy Lie--Schouten bracket on $\Pi L^*$\,. 

Same works for  $\Linf$-algebroids.

What about morphisms?  The case of $\Linf$-algebroids is, roughly, a combination of the cases of the usual Lie algebroids and $\Linf$-algebras. So let us recall the notion of morphisms of $\Linf$-algebras, the so called   \emph{$\Linf$-morphisms} (Lada--Stasheff~\cite{lada:stasheff}). If two $\Linf$-algebras $L$ and $K$ are given, an $\Linf$-morphism $L\specto K$ not a map from $L$ to $K$ in any sense (that is why we use a special arrow). An algebraic definition given in~\cite{lada:stasheff} is that  an \emph{$\Linf$-morphism} $L\specto K$ is a sequence of linear maps
\begin{equation*}
    \La^n L\to K
\end{equation*}
(equivalently, skew-symmetric multilinear maps $L\times \ldots \times L\to K$) of alternating parities (so that the linear map $L\to K$ in this sequence is even) and satisfying an infinite sequence of identities  for inserting brackets to multilinear maps and other way round. We skip their exact form. A remarkable fact is that all these identities can be packed into one, namely, that a formal map $\Pi L\to \Pi K$ obtained by assembling together the maps $S^n(\Pi L)\to \Pi K$ (recall $S^n(\Pi L)\cong \Pi^n(\La^n L)$) is a $Q$-map for the corresponding homological vector fields. 

We  summarize this in a \hypertarget{tab.manisec}{table} (partly incomplete).

\begin{center} 
{\renewcommand{\arraystretch}{1.3}
\begin{tabular}{|c|c|} \hline
  \multicolumn{2}{|c|}{\rule{0pt}{15pt} Manifestations of $\Li$-algebras }\\
  \hline
  \rule{0pt}{15pt} Objects: $L$ & Morphisms: $L\specto K$  \\ \hline
  $L$ with antisymmetric brackets   & $\L^nL\to K$ ($n=0,1,2,\ldots$ )   \\
$V=\Pi L$ with symmetric odd brackets  & $S^nV\to W$ ($n=0,1,2,\ldots$ ) \\
$V$ as formal $Q$-manifold &   $Q$-map $V\to W$\\
$\Pi L^*=V^*$ as $\Sinf$-manifold & ??\\
$L^*=\Pi V^*$ as $\Pinf$-manifold & ??\\
\rule{0pt}{-20pt} &  \rule{0pt}{-20pt}\\
   \hline
\end{tabular}
}
\end{center}

We stress the following: since for an $\Linf$-morphism $L\specto K$ there is \emph{no} linear (or any other) map from $L$ to $K$ and the map $\Pi L\to \Pi K$ is non-linear, it is not possible, in a conventional way, as we do it for homomorphisms of ordinary Lie algebras, to obtain   maps in the opposite direction, from $K^*$ to  $L^*$  or from $\Pi K^*$ to  $\Pi L^*$ (which in the ordinary case are also  Poisson maps, mapping the bracket of functions  on $(\Pi)L^*$ to the bracket of functions on $(\Pi)K^*$). We shall solve this riddle in the next section, together with the problem concerning higher Koszul brackets.

\section{Higher Koszul brackets. Problem and   solution}\label{sec.koszul}

Here we define higher Koszul brackets, 
state the problem and describe the solution.

\subsection{Classical Koszul bracket} 

Let $M$ be a supermanifold. Consider first the case when $M$ is endowed with a usual Poisson structure (before going to the homotopy case). Recall the construction of Koszul bracket induced in the algebra of forms $\O(M)$. Though the original Koszul's definition was for ordinary manifolds, it is easy to   define the bracket in the super case. If we denote a given Poisson bracket on functions as $\{f,g\}_P$, where $f,g\in \fun(M)$, then the Koszul bracket, denoted is $[-,-]_P$, is  defined on functions and differentials of function by the formulas:
\begin{equation}\label{eq.initkoszulbin}
    [f,g]_P:=0\,, \quad [f,dg]_P=(-1)^{\ft}\{f,g\}_P\,, \quad [df,dg]_P=  -(-1)^{\ft}d\{f,g\}_P\,,
\end{equation}
and is extended to the whole $\O(M)$ by the Leibniz rule. More precisely, it is easy to check that the odd brackets~\eqref{eq.initkoszulbin} are bilinear and  symmetric with respect to reversed parity.
Then one can check that the Jacobi identity in the form~\eqref{eq.oliejac} is satisfied for combinations of $f$'s and $df$'s and then conclude that extension by the Leibniz rule makes $\O(M)$ an odd Poisson algebra.

An alternative approach is by introducing a master Hamiltonian for Koszul bracket. We only give a formula (found in~\cite{tv:higherpoisson}) because in the next subsection we consider more general case. In coordinates $(x^a,dx^a,p_a,\pi_a)$ on $T^*(\Pi TM)$ we have
\begin{equation}\label{eq.masterforbinkosz}
    K = -P^{ab}\pi_bp_a +\frac{1}{2}\,dP^{ab}\pi_b\pi_a\,,
\end{equation}
if the Poisson bracket on $M$ is given by $P=\frac{1}{2}\, P^{ab}x^*_bx^*_a$\,.

\begin{proposition} The odd Hamiltonian $K$ is well-defined. It satisfies $(K,K)=0$ and the odd Poisson bracket defined in $\O(M)$ by 
\begin{equation}\label{eq.koszulbinder}
    [\o,\s]=((K,\o),\s)
\end{equation}
coincides with the bracket $[-,-]_P$ given by~\eqref{eq.initkoszulbin}. 
\end{proposition}

Notice now that this classical Koszul bracket can be seen as a Lie algebroid structure in $T^*M$. Indeed, it is an odd Poisson bracket on the vector bundle $\Pi TM$  of weight $-1$, so by general theory explained in the previous section, corresponds to a Lie algebroid structure in $T^*M$. The bracket of $1$-forms is  a bracket of sections (of the parity-reversed bundle $\Pi T^*M$) and the anchor $a_P$ is given by the bracket
\begin{equation*}
    a_P(fdg)(h)= [fdg,h]_P=f\{g,h\}_P\,.
\end{equation*}
We immediately see that as a map $a_P\co \Pi T^*M\to \Pi TM$ it is given by the formula
\begin{equation}\label{eq.anchkoszbin}
    a_P^*(dx^a)=P^{ab}x^*_b\,,
\end{equation}
so the pullback by $a_P$  is exactly the `raising indices' map $P^{\#}\co \O(M)\to \Mult(M)$, $P^{\#}=a_P^*$. The anchor $a_P\co \Pi T^*M\to \Pi TM$ is in particular a morphism of vector bundles over a fixed base. Consider its dual or adjoint, which should be a morphism of the dual bundles in the opposite direction. But $(\Pi TM)^*=\Pi T^*M$ and $(\Pi T^*M)^*=\Pi TM$, so the dual $a_P^*$ (no confusion with pullback!) is a morphism of the same bundles: $a_P^*\co \Pi T^*M\to \Pi TM$. It we  calculate it, we obtain that it is given by `raising indices' with the help of $T^{ab}=P^{ba}(-1)^{(\at+1)(\bt+1)}$. But due to the symmetry of the Poisson tensor $P^{ab}$ we have $T^{ab}=P^{ab}$, hence the dual map $a_P^*$ equals $a_P$.

Let us summarize. The structure of a Lie algebroid induced into $T^*M$ by a Poisson structure on $M$ has   equivalent manifestations: as the $Q$-structure, it is the Lichnerowicz differential regarded as a vector field $d_P\in \Vect(\Pi T^*M)$; as the Lie--Schouten bracket, it is the Koszul bracket on $\Pi TM$. (There is also another manifestation, the Lie--Poisson bracket on $TM$, which we do not consider now.) By general theory, for any Lie algebroid the anchor is a Lie algebroid morphism. So in our case, we have a Lie algebroid morphism $a_P\co T^*M\specto TM$. Its two manifestations are  the vector bundle map  $a_P\co \Pi T^*M\to \Pi TM$  and its dual $a_P^*$, which  in our case coincide. Hence we have for free the classical statement:
\begin{theorem} For a Poisson (super)manifold $M$ with a Poisson tensor $P$, the operation of  raising indices by $P$, $P^{\#}\co \O(M)\to \Mult(M)$, maps the de Rham differential to the Lichnerowicz differential and the Koszul bracket to the Schouten bracket. \qed
\end{theorem}

\begin{remark}
Let us stress that   this formulation combines two manifestations (on dual bundles) for each of    the two Lie algebroid structures: the canonical one in   $TM$ and the one in   $T^*M$ induced by a Poisson bracket on $M$, and of the anchor  $T^*M\specto TM$ linking them.  This explains why the canonical structures and the structures depending on $P$ meet ``cross-wise'':
\begin{equation*}
    \text{(de Rham, Koszul)} \to \text{(Lichnerowicz, Schouten)}\,.
\end{equation*}
The ``accidental equality'' of the vector bundle morphism $a_P$ and its dual arises because of the (anti)symmetry $P^{ab}=(-1)^{(\at+1)(\bt+1)} P^{ba}$.
\end{remark}

Now we can turn to the homotopy case.

\subsection{Definition of higher Koszul brackets} 

Let a supermanifold has a $\Pinf$-structure (homotopy Poisson). It is specified by a `master antiHamiltonian' (or multivector field) $P\in \fun(\Pi T^*M)$, which is odd and satisfies $\lsch P,P,\rsch=0$ (equivalently, $[P,P]=0$, in alternative convention for the Schouten bracket)\,. It defines the homotopy brackets of functions by~\eqref{eq.pinfder}.

The \emph{higher Koszul brackets} are defined on forms on $M$ by the formulas for functions and differentials of functions~\cite{tv:higherpoisson}
\begin{align}
    [f]_P&=\{f\}_P \quad \text{and} \quad [f_1,\ldots,f_k]=0 \ \text{for $k\geq 2$}\,,  \label{koszul.eq.higherfunc}\\
    [f_1,df_2,\ldots,df_n]_P&=(-1)^{\e}\,\{f_1,f_2,\ldots,f_n\}_P\,,  \label{koszul.eq.higherfuncdif}\\
    [df_1,\ldots,df_n]_P&=(-1)^{\e+1}\,d\{f_1,\ldots,f_n\}_P\,,  \label{koszul.eq.higherdifdif}
\end{align}
where $\e=(n-1)\ft_1+(n-2)\ft_2+\ldots+\ft_{n-1}+n$. There are no other non-zero brackets between $f$'s and $df$'s apart of   obtained from the above by symmetry. To the whole algebra $\O(M)$ the brackets are extended by the Leibniz rule. The above formulas are very clearly see as specifying an $\Linf$-algebroids structure in $T^*M$, namely, the brackets of sections and higher anchors.

It is very easy to recognize the whole general framework for higher Koszul brackets. Since there is a homotopy Poisson structure on $M$, specified by a master Hamiltonian $P$, there arises its ``Hamiltonian vector field'' that lives on $\Pi T^*M$. By definition, it is the \emph{Lichnerowicz differential} for a $\Pinf$-structure:
\begin{equation}\label{eq.lichnerforpinf}
    d_P:=[P, -]\,.
\end{equation}
This is exactly the $Q$-  manifestation of an $\Linf$-algebroid structure that arises in $T^*M$. By the formula for the canonical Schouten bracket,
\begin{equation}\label{eq.anchoronpitstarm}
    d_P=  (-1)^{\at}\der{P}{x^*_a}\der{ }{x^a} + (-1)^{\at}\der{P}{x^a}\der{ }{x^*_a}\,.
\end{equation}

From it we immediately obtain the anchor (as a map $\Pi T^*M\to \Pi TM$). It is given by the first term in~\eqref{eq.anchoronpitstarm}. We have a fiberwise  map  
$a_P\co \Pi T^*M\to \Pi TM$ over $M$, $(x^a,x^*_a)\mapsto (x^a, dx^a=(-1)^{\at}\lder{P}{x^*_a})$ (this differs by the negative sign from~\cite{tv:higherpoisson}, which plays no role).  For quadratic $P$ (the classical case) it immediately gives `raising indices' by $P^{ab}$. Since the anchor for any $\Linf$-algebroid is an $\Linf$-morphism (a generalization of the classical fact for Lie algebroids), we without effort get the theorem
\begin{theorem} The diagram
\begin{equation}
    \begin{CD} \Mult(M)@>{d_P}>> \Mult (M)\\
                @A{a_P^*}AA         @AAa_P^*A\\
                \O(M)@>{d}>> \O(M) \,,
    \end{CD}
\end{equation}
is    commutative. Here the vertical arrows are pullbacks by the   map $(x^a,x^*_a)\mapsto (x^a, dx^a=(-1)^{\at}\lder{P}{x^*_a})$. \qed
\end{theorem}

Now we see the \textbf{problem}. We already stated it in the Introduction as the question how to relate the sequence of higher Koszul brackets in the algebra of forms (which make $\O(M)$ a $\Sinf$-algebra) with the canonical Schouten bracket of multivector fields. In the classical situation, when there was just one binary Koszul bracket, there was simply a homomorphism of Lie superalgebras mapping it to the Schouten bracket. As we established, it was the pullback of functions with respect to the adjoint for the linear map $a_P\co \Pi T^*M \to \Pi TM$. By the property of the Poisson tensor, the adjoint to the anchor coincided with the anchor, so effectively the homomorphism between Koszul bracket and the Schouten bracket was the same map that intertwined $d$ and $d_P$. Now this would not work for the simple reason that $a_P$ is a non-linear map and it is clear how to take its adjoint to begin with, and what it could be. 

These are exactly the questions we need to answer to be able to complete the table at the end of the previous section.

We can give now a sufficiently explicit answer to the question about higher Koszul brackets and then explain why it works.

\begin{theorem} 
\label{thm.solu}
The following non-linear transformation of \underline{even} differential forms on $M$ to \underline{even} multivector fields is an $\Linf$-morphism from the $\O(M)$ regarded as an   $\Linf$-algebra with respect to the higher Koszul brackets to the odd Lie superalgebra $\Mult(M)$  with the canonical Poisson bracket: an even form $\o=\o(x,dx)$ is mapped to a multivector field $\s=\s(x,x^*)$ given by the formula
\begin{equation}\label{eq.solution}
    \s(x,x^*)= \o(x,dx) + \der{P}{x^*_a}\left(x, \der{\o}{dx}(x,dx)\right) x^*_a -dx^a\der{\o}{dx^a}(x,dx)\,,
\end{equation}
where the variables $dx^a$ are eliminated from right-hand side by substituting the solution of the non-linear equation
\begin{equation}\label{eq.solution2}
    dx^a= \dder{P}{x^*_a}{x^*_b}\left(x, \der{\o}{dx}(x,dx)\right)x^*_b\,.
\end{equation}
Here $P$ is the Poisson tensor. 
\end{theorem}
(Some signs may need checking.)

Here in the statement we used the manifestation of $\Linf$-morphisms in terms of $Q$-manifolds: in our case, the ``manifolds'' are infinite-dimensional supermanifolds whose ``points'' are even forms and even multivector fields on $M$, respectively. To obtain an algebraic description linking brackets as such, one has to develop the constructed non-linear map of forms to multivectors in a Taylor series.  

Of course, formulas such as~\eqref{eq.solution} and~\eqref{eq.solution2}  are not elucidating unless their meaning is clarified. Even if these formulas are   given, there is no hope to prove the theorem by brute force, not knowing where they  come from. We shall explain it now.

\subsection{Thick morphisms. Their applications to vector bundles and algebroids}\label{ssec.thick}
If we want to follow the logic that applies in the ordinary case, we would need the adjoint   map for the anchor map $a_P\co \Pi T^*M\to \Pi TM$ given by
\begin{equation*}
    dx^a=(-1)^{\at}\der{P}{x^*_a}\,;
\end{equation*}
unfortunately, such an adjoint cannot exist because the unless for a quadratic $P$, the map $a_P\co \Pi T^*M\to \Pi TM$ is not fiberwise linear.
Nevertheless, \textbf{adjoints of non-linear maps exist if we enlarge the class of morphisms}.  This was discovered in~\cite{tv:nonlinearpullback,tv:microformal}. The corresponding new morphisms of supermanifolds are called \emph{thick morphisms} or \emph{microformal morphisms}.

The key feature of thick morphisms is the construction of pull-back that (as in the usual case) maps functions in the opposite direction, but unlike the  usual case, the pullback by a thick morphism is a non-linear mapping of functions. Non-linear pullbacks are of course a radical departure from the familiar picture where pullbacks are algebra homomorphisms. It turns out that in spite of non-linearity, these pullbacks still possess some important algebraic properties. Among them: that the derivative of a pullback at each element of the space of functions is an algebra homomorphism. We will not retell here the theory of thick morphisms and non-linear pullbacks, referring the reader to ~\cite{tv:nonlinearpullback,tv:microformal}. We just quote the constructions and statements needed for our purpose.

First of all, the non-linearity of pullbacks forces to distinguish between even and odd functions (because of different commutativity properties). They make, respectively, infinite-dimensional supermanifolds denoted $\funn(M)$ and $\pfunn(M)$. (`Points' of $\funn(M)$ are bosonic fields, `points' of $\pfunn(M)$ are fermionic fields.) There arise two parallel theories: for even and odd functions. 

In bosonic theory, an     \emph{even thick morphism}\footnote{\,The adjective concerns the parity of functions on which the pullback by $\F$ is defined.} $M\F^*\co M_1\tto M_2$ is a formal canonical relation  from $T^*M_1$ to $T^*M_2$, i.e. a formal Lagrangian submanifold  in $T^*M_2\times (-T^*M_1)$ with respect to the symplectic form $\o_2-\o_1$, which is specified by an even generating function $S=S(x,q)$, where $x^a$ are coordinates on the source, $p_a$ are the conjugate momenta, $y^i$ are coordinates of the target and $q_i$ are the conjugate momenta.  The crucial construction of $\F^*\co \funn(M_2)\to \funn(M_1)$ is given by the following formula:

\begin{equation}\label{eq.pullbackev}
    \boxed{ \Phi^*[g] (x)= g(y) + S(x,q) - y^iq_i\,, \vphantom{\der{S}{q_i}}}
\end{equation}
for $g\in \funn(M_2)$, where $q_i$ and $y^i$ are determined from the equations
\begin{equation}
   { q_i=\der{g}{y^i}\,(y)\,, \quad y^i=(-1)^{\itt}\,\der{S}{q_i}(x,q)  }
\end{equation}
(which gives $y^i=(-1)^{\itt}\der{S}{q_i}(x,\der{g}{y}(y))$ solvable by iterations).

A generating function $S$ is considered as a formal expansion near the zero section:
\begin{equation*}
    S(x,q)=S_0(x)+ S^i(x)q_i + \frac{1}{2}\,S^{ij}(x)q_jq_i + \frac{1}{3!}\,S^{ijk}(x)q_kq_jq_i +\ldots
\end{equation*}
Ordinary maps $\f\co M_1\to M_2$ correspond to $S=\f^(x)q_i$ (linear in the target momenta) and for them the pullback is the usual pullback.

A parallel theory uses anticotangent bundles $\Pi T^*M$ and their odd symplectic structure. For distinction, \emph{odd thick morphisms} acting on odd functions are denoted $\Psi\co M_1\oto M_2$. The pullback $\Psi^*\co \pfunn(M_2)\to \pfunn(M_1)$ is constructed similarly with the above, by using an odd generating function $S=S(x,y^*)$ depending on source coordinates and target antimomenta.

Two fundamental properties of even and odd thick morphisms are relevant for us.

First, is their relation with homotopy Poisson structures. 

Let $M_1$ and $M_2$ be $S_{\infty}$-manifolds, with  master Hamiltonians  $H_i\in \fun(T^*M_i)$, $i=1,2$.
\begin{definition} A thick morphism $\Phi\co M_1\tto M_2$ is  an  \emph{$S_{\infty}$  thick morphism }
if $\pi_1^*H_1=\pi_2^*H_2$ on $\F$ regarded as a submanifold in $T^*M_2\times T^*M_1$.
\end{definition}
This is expressed by the Hamilton--Jacobi equation for $S(x,q)$
  \begin{equation}\label{eq.schoutenhj}
    H_1\Bigl(x,\der{S}{x}\Bigr)=H_2\Bigl(\der{S}{q},q\Bigr)\,.  
\end{equation}

\begin{theorem} \label{thm.sinfthick}
 If a thick  morphism of  $S_{\infty}$-manifolds $\Phi\co M_1\tto M_2$ is  {$S_{\infty}$}, then the pullback
\begin{equation*}
    \Phi^*\co \funn(M_2)\to \funn(M_1)
\end{equation*}
is an $L_{\infty}$-morphism of the homotopy Schouten brackets.
\end{theorem}

There is a parallel statement for $\Pinf$ odd thick morphisms defined similarly. 

Secondly, there is the notion of the adjoint for any non-linear map of vector spaces or vector bundles, if instead of ordinary maps we consider thick morphisms. (To achieve complete symmetry, this construction extends to thick morphisms themselves.) Namely,  if $\Phi\co E_1\to E_2$
is a  fiberwise (in general nonlinear) map of vector bundles,
then there is   a   thick morphism  called fiberwise \emph{adjoint} or \emph{dual}
\begin{equation*}
    \Phi^*\co E_2^*\tto E_1^*\,,
\end{equation*}
coinciding with    the usual adjoint map if $\Phi$ is fiberwise-linear, and with the same properties. It is defined as follows:
\begin{equation*}
    \Phi^*:=\bigl( \kir\times \kir)({\F} \bigr)^{\text{op}}\subset T^*E^*_1\times (-T^*E^*_2)\,,
\end{equation*}
where $\kir\co T^*E\to T^*E^*$ is the Mackenzie--Xu diffeomorphism. Practically, on the level of generating functions, it is an exchange up to signs of fiber coordinates and conjugate momenta for fiber coordinates. There is a parallel `odd adjoint' or `odd dual' construction. 
\begin{equation*}
    \Phi^{*\Pi}\co \Pi E_2^*\oto \Pi E_1^*\,,
\end{equation*}
coinciding with    the  usual adjoint combined with parity reversion if $\Phi$ is fiberwise-linear, and with the same properties. (Construction uses the odd analog of Mackenzie--Xu diffeomorphism found in~\cite{tv:graded}.)

Suppose we have $\Linf$-algebroids $E_1$ and $E_2$ over the same base $M$. An $\Linf$-morphism $\F\co E_1\specto E_2$ can be described as a fiberwise $Q$-map $\F\co \Pi E_1\to \Pi E_2$. Although it is not fiberwise-linear in general, for $\F$ we can consider its adjoint as a (fiberwise) thick morphism $\F\co \Pi E_2^*\to \Pi E_1^*$\,. The condition that $\F\co \Pi E_1\to \Pi E_2$ is a $Q$-map means that the homological vector fields $Q_1$ and $Q_2$ on $\Pi E_1$ and $\Pi E_2$ are $\F$-related. Hence $\F$-related are the corresponding linear Hamiltonians $H_{Q_1}$ and $H_{Q_2}$. Taking the adjoint in this language basically means applying the Mackenzie--Xu transformation to the  product of the cotangent bundles and the flip of factors. Hence the $\F$-related Hamiltonians transform into $\F^*$-related. It remains to note that the Mackenzie--Xu transformation of the original linear Hamiltonians gives exactly the master Hamiltonians of the Lie--Schouten brackets. Therefore we can apply Theorem~\ref{thm.sinfthick}. Thus we have arrived at the following statement:

\begin{theorem} 
An $\Linf$-morphism $\F\co E_1\specto E_2$ of $\Linf$-algebroids over a fixed base induces a $\Sinf$ thick morphism $\Pi E_2^*\tto \Pi E_1^*$ and an $\Linf$-morphism of Lie--Schouten brackets (as a non-linear map of functions $\F_*\co \funn(\Pi E_1^*)\to \funn(\Pi E_2^*)$).
\end{theorem}

There is an analog for homotopy Lie--Poisson brackets. 

Now we can fill the gaps in the \hyperlink{tab.manisec}{table} at the end of Section~\ref{sec.recall}. We obtain the following (where we write for the case of $\Linf$-algebroids).

\begin{center}
{\renewcommand{\arraystretch}{1.3}
\begin{tabular}{|c|c|} \hline
  \multicolumn{2}{|c|}{\rule{0pt}{15pt} Manifestations of $\Li$-algebroids }\\
  \hline
  \rule{0pt}{15pt} Objects: $E$ & Morphisms: $E_1\specto E_2$  (fixed base)\\ \hline
  $E$ with antisymmetric brackets   & $\La^nE_1\to E_2$ ($n=0,1,2,\ldots$ )   \\
$F=\Pi E$ with symmetric odd brackets  & $S^nF_1\to F_2$ ($n=0,1,2,\ldots$ ) \\
$F$ as formal $Q$-manifold &   $Q$-map $F_1\to F_2$\\
\rule{0pt}{15pt}$\Pi E^*=F^*$ as $\Sinf$-manifold & $\Sinf$ thick morphism $\Pi E_2^* \tto \Pi E_1^*$\\
  & induces $\Linf$-morphism $\funn(\Pi E_1^*) \to \funn(\Pi E_2^*)$\\
\rule{0pt}{15pt}$E^*=\Pi F^*$ as $\Pinf$-manifold & $\Pinf$ thick morphism $E_2^* \oto E_1^*$\\
  & induces $\Linf$-morphism $\pfunn(E_1^*) \to \pfunn(E_2^*)$\\
\rule{0pt}{-20pt} &  \rule{0pt}{-20pt}\\
   \hline
\end{tabular}
}
\end{center}

Now, to obtain the formulas in Theorem~\ref{thm.solu}, consider the anchor $a_P\co \Pi T^*M\to \Pi TM$ for  $T^*M$ given by   $(x^a,x^*_a)\mapsto (x^a, dx^a=(-1)^{\at}\lder{P}{x^*_a})$. Its generating function is
\begin{equation*}
    S=S(x,x^*; p, \pi)=x^ap_a + \der{P}{x^*_a}(x,x^*)\pi_a\,.
\end{equation*}
For the adjoint, we take the ``dual'' generating function:
\begin{equation*}
    S^*=S^*(x,x^*;p,\pi)=x^ap_a + \der{P}{x^*_a}(x,\pi)x^*_a\,.
\end{equation*}
Formulas~\eqref{eq.solution} and \eqref{eq.solution}   follow directly by applying the construction of pullback.

\begin{remark} It is possible also to include the whole construction of higher Koszul brackets (and considerations here) into a general framework of (quasi)triangular $\Linf$-bialgebroids. We hope to do that separately.  
\end{remark}

\appendix
\section{BV operator for higher Koszul brackets}

The higher Koszul brackets for a homotopy Poisson structure can be obtained from a BV-type operator on forms in a way similar to what it was done by Koszul in~\cite{koszul:crochet85}. (Thanks are due to Jim Stasheff for asking whether it is possible and to Yvette Kosmann-Schwarzbach for reminding Koszul's original construction.)  However, there are differences.

The main one is that, unlike the classical case where the generating operator is of second order and there is only one  bracket considered, which is binary, and it is its (quadratic) principal symbol, here we have to deal with an operator whose order is potentially not restricted (i.e. a formal differential operator), and the entirety of brackets (more precisely, their master Hamiltonian, which is inhomogeneous) should be interpreted as its ``principal symbol'' in the sense of $\hbar$-differential operators (which is different). This involves $\hbar$ as a parameter and taking the limit $\hbar \to 0$.

The details are as follows.

Let $P \in C^{\infty}(\Pi T^*M)$ be a ``Poisson tensor", i.e an even function satisfying $\lsch P,P\rsch=0$ (for the canonical Schouten bracket). In the ordinary manifold case, it is a sum of terms of various degrees (an inhomogeneous multivector field); in the supercase, it is an arbitrary smooth function $P=P(x,x^*)$, $x^*_a$ being the antimomenta (though only its expansion at $x^*=0$ matters). Then we set $\Delta$ acting on forms on $M$ (treated as functions $\omega=\omega(x,dx)$) to be
\begin{equation}
    \Delta := [d, \hat P],
\end{equation}
where for $P=P(x,x^*)$,
\begin{equation}
     \hat P := P(x, \frac{\hbar}{i} \der{}{dx})\,.
\end{equation}
(This is the analog of the interior product by a multivector. Basically, we interpret $P$ as a differential operator on forms, with differentiation only with respect to $dx^a$. Note the appearance of $\hbar$. One can conveniently express this operator by a Berezin integral.)

\begin{theorem}
For $\Delta$ as above, the higher Koszul brackets on forms  corresponding to a homotopy Poisson structure specified by $P$  are given by
\begin{equation}
    [\omega_1,...,\omega_k]_{P} = \lim_{\hbar\to 0}\, \bigl[...[\Delta,\omega_1],...,\omega_k\bigr] (1).
\end{equation}
\end{theorem}
(These are precisely the ``classical brackets" generated by an operator $\Delta$ as defined in~\cite[\S 5]{tv:microformal}.)
\begin{proof}
We need to show that the odd master Hamiltonian for the sequence of higher Koszul brackets is the principal symbol of the operator $\Delta$. For that we use formula (5.11) for the principal symbol from~\cite{tv:microformal}
(written for functions on $\Pi TM$)
\begin{equation*}
    H(x,dx,p,\pi) = \lim_{\hbar\to 0} e^{- \frac{i}{\hbar} (xp + dx \pi)} \Delta (e^{\frac{i}{\hbar} (xp + dx \pi)})
\end{equation*}
(analogous to the well familiar formula in the theory of pseudodifferential operators). Here $x^a, dx^a$ are local coordinates on $\Pi TM$ and $p_a,\pi_a$ denote the conjugate momenta. We directly calculate the action of the commutator $[d, \hat P]$ on the exponential. At one stage, expressing the action of the operator $\hat P$ is helpful. At the end, we obtain (before taking the limit):
\begin{equation*}
    e^{- \frac{i}{\hbar} (xp + dx \pi)} \Delta (e^{\frac{i}{\hbar} (xp + dx \pi)})= dx^a \der{P}{x^a}(x,\pi) + (-1)^{\tilde a} \der{P}{\pi_a}p_a
+ \frac{\hbar}{i}P(x,\pi) dx^a p_a\,.
\end{equation*}
After passing to the limit $\hbar\to 0$, we arrive at the Hamiltonian
\begin{equation*}
    dx^a \der{P}{x^a}(x,\pi) + (-1)^{\tilde a} \der{P}{\pi_a}p_a\,,
\end{equation*}
which exactly coincides with the master Hamiltonian $K$ for the higher Koszul brackets found in~\cite{tv:higherpoisson}
(formula (24) in Sect. 4 there). This completes the proof.
\end{proof}


  \def\cprime{$'$} \def\cprime{$'$}

\end{document}